\documentclass[10pt,a4paper]{amsart}

\usepackage{amssymb,mathrsfs,graphicx}
\usepackage[english]{babel}
\usepackage[leqno]{amsmath}
\usepackage{amssymb,amsthm,amsmath}
\usepackage{enumerate}
\usepackage{epsfig}

\input{xy}
\xyoption{all}

\newtheorem{thm}{Theorem}[subsection]
\newtheorem{theoremalpha}{Theorem}

\newtheorem{lemma}[thm]{Lemma}
\newtheorem{lemmadefi}[thm]{Lemma - Definition}

\newtheorem{prop}[thm]{Proposition}

\newtheorem{cor}[thm]{Corollary}

\newtheorem{fact}[thm]{Fact}

\theoremstyle{remark}
\newtheorem{remark}[thm]{Remark}
\theoremstyle{definition}
\newtheorem{defi}[thm]{Definition}

\numberwithin{equation}{section}

\newenvironment{sis}{\left\{\begin{aligned}}{\end{aligned}\right.}

\newtheorem{example}[thm]{Example}

\newcommand{\ov}{\overline}
\newcommand{\un}{\underline}

\renewcommand{\Im}{\operatorname{Im}}

\newcommand{\val}{\operatorname{val}}

\newcommand{\rk}{\operatorname{rk}}

\newcommand{\Jac}{\operatorname{Jac}}

\newcommand{\GL}{\operatorname{GL}}

\newcommand{\Q}{\mathbb{Q}}
\newcommand{\Z}{\mathbb{Z}}
\newcommand{\R}{\mathbb{R}}

\newcommand{\bbD}{\mathbb{D}}

\newcommand{\CC}{\mathcal{C}}
\newcommand{\DD}{\mathcal{D}}
\newcommand{\BB}{\mathcal{B}}
\newcommand{\II}{\mathcal{I}}

\newcommand{\D}{\mathcal{D}}

\newcommand{\Del}{{\rm Del}}
\newcommand{\Vor}{{\rm Vor}}

\newcommand{\prin}{\sigma_{{\rm prin}}}

\renewcommand{\O}{\Omega_g}
\newcommand{\Ort}{\Omega_g^{\rm rt}}

\newcommand{\Omat}{\Omega_g^{\rm mat}}
\newcommand{\Null}{{\rm Null}}

\newcommand{\Mg}{\mathcal{M}_g}
\newcommand{\Mgb}{\ov{\mathcal{M}_g}}
\newcommand{\Ag}{\mathcal{A}_g}

\newcommand{\tg}{{\rm t}_g}
\newcommand{\tgb}{{\ov{{\rm t}_g}}}

\newcommand{\Cen}{\Sigma_{\operatorname{C}}}
\newcommand{\Per}{\Sigma_{\operatorname{P}}}
\newcommand{\Voro}{\Sigma_{\operatorname{V}}}
\newcommand{\Mat}{\Sigma_{{\rm mat}}}

\newcommand{\AgVor}{\ov{\Ag}^V}
\newcommand{\AgPer}{\ov{\Ag}^P}
\newcommand{\Agmat}{\ov{\Ag}^{\rm mat}}
\newcommand{\AgCen}{\ov{\Ag}^C}
\newcommand{\Agcogra}{\ov{\mathcal{A}}_g^{\rm cogr}}

\begin{document}

\title[Comparing Perfect and 2nd Voronoi]{Comparing Perfect and 2nd Voronoi decompositions: the matroidal locus}

\author{Margarida Melo and Filippo Viviani}
\address{Departamento de Matem\'atica da Universidade de Coimbra,
Largo D. Dinis, Apartado 3008, 3001 Coimbra (Portugal) \text{ and } Dipartimento di Matematica,
Universit\`a Roma Tre,
Largo S. Leonardo Murialdo 1,
00146 Roma (Italy)}
\email{mmelo@mat.uc.pt, viviani@mat.uniroma3.it}

\thanks{The first author  was supported by the FCT project \textit{Espa\c cos de Moduli em Geometria Alg\'ebrica} (PTDC/MAT/111332/2009), by the FCT project \textit{Geometria Alg\'ebrica em Portugal} (PTDC/MAT/099275/2008) and by the Funda\c c\~ao Calouste Gulbenkian program ``Est\'imulo \`a investiga\c c\~ao 2010''.
The second author was supported by the grant FCT-Ci\^encia2008 from CMUC (University of Coimbra)
and by the FCT project \textit{Espa\c cos de Moduli em Geometria Alg\'ebrica} (PTDC/MAT/111332/2009).}




\begin{abstract}
We compare two rational polyhedral admissible decompositions of the cone of positive definite quadratic forms: the perfect cone decomposition and the 2nd Voronoi decomposition. We determine which cones belong to both the
decompositions, thus providing a positive answer to a conjecture of Alexeev-Brunyate  in \cite{AB}.
As an application, we compare the two associated toroidal compactifications of the moduli space of principal polarized abelian varieties: the perfect cone compactification and the 2nd Voronoi compactification.
\end{abstract}

\maketitle



\section{Introduction}

The theory of reduction of positive definite quadratic forms consists in finding a fundamental domain for the natural action of $\GL_g(\Z)$ on the cone
$\O$ of positive definite quadratic forms of rank $g$ or, more generally, on its rational closure $\Ort$, i.e. the cone of positive semi-definite quadratic forms whose null
space is defined over the rationals.
 One way to achieve this is to find a decomposition of the cone $\Ort$  into an infinite
$\GL_g(\Z)$-periodic face-to-face collection of rational polyhedral subcones (or, in short, an admissible decomposition, see Definition \ref{decompo} for details) in such a way that there are only
finitely many $\GL_g(\Z)$-equivalence classes of subcones.
This theory is very classical, dating back to work of Minkowsky \cite{Min}, Voronoi \cite{Vor} and Koecher \cite{Koe}.

A renewed interest in this theory came when Ash-Mumford-Rapoport-Tai (see \cite{AMRT}) showed how to associate to every admissible decomposition of $\Ort$ a compactification of the moduli
space $\Ag$ of principally polarized abelian varieties of dimension $g$, a so-called toroidal compactification of $\Ag$. See also the book of Namikawa \cite{NamT} for a nice account of the theory.

The aim of this paper is to compare two well-known  admissible decompositions of $\Ort$ (both introduced by Voronoi in \cite{Vor}), namely:
\begin{enumerate}[(i)]
\item The \emph{perfect cone decomposition} $\Per$ (also known as the first Voronoi
decomposition);
\item The \emph{2nd Voronoi decomposition} $\Voro$ (also known as  the L-type decomposition).
\end{enumerate}
We refer to Sections \ref{S:Perfect} and Sections \ref{S:Voro} for the definitions of the above admissible decompositions.

Consider the toroidal compactifications of $\Ag$ associated to the perfect and the 2nd Voronoi decompositions: the perfect toroidal compactification and the 2nd Voronoi toroidal compactification, respectively. Denote them  by $\AgPer$ and by $\AgVor$, respectively.
Each of these compactifications plays an important role in the theory of the compactifications of $\Ag$:
\begin{enumerate}[(i)]
\item $\AgPer$ is the canonical model of $\Ag$ for $g\geq 12$ (Shepherd-Barron \cite{SB}).
\item $\AgVor$ is (up to possibly normalizing) the main irreducible component of Alexeev's
    moduli space $\ov{AP_g}$ of stable semiabelic pairs, which provides a \emph{modular} compactification of $\Ag$
    (Alexeev \cite{alex1}). See also the work of Olsson \cite{Ols} for a different modular interpretation of $\AgVor$ via logarithmic geometry.
\end{enumerate}
Moreover, each of them is well-suited to compactify the Torelli map. Indeed,
the Torelli map
$$ \tg : \Mg\to \Ag,  $$
sending a curve $X\in \Mg$ into its polarized Jacobian $(\Jac(X), \Theta_X)\in \Ag$, extends to regular maps
\begin{equation*}
\tgb^V: \Mgb\to \AgVor \hspace{0,3cm} \text{ and } \hspace{0,3cm} \tgb^P: \Mgb\to \AgPer,
\end{equation*}
where $\Mgb$ is the Deligne-Mumford (see \cite{DM}) compactification of $\Mg$ via stable curves. The existence of
$\tgb^V$ is classically due to Mumford-Namikawa \cite{nam2} (see also Alexeev \cite{alex} for a modular interpretation). For a long period, this was the only known compactification of the Torelli map
until the recent breakthrough work of Alexeev-Brunyate \cite{AB} who proved the existence of the regular map
$\tgb^P$. Moreover, Alexeev-Brunyate also showed in loc. cit. that  $\AgPer$ and $\AgVor$ are isomorphic on an open subset containing the image of $\Mgb$ via the compactified Torelli maps
$\tgb^P$ and $\tgb^V$, namely the cographic locus (see Fact \ref{F:comp-Tor} for more details).
Further, they indicate in \cite[6.3]{AB} a bigger open subset where $\AgPer$ and $\AgVor$
should be isomorphic, namely the matroidal locus (see Definition \ref{D:mat-locus}).
The aim of this paper, which was very much inspired by the reading of \cite{AB}, is to give a positive answer to their conjecture and, moreover, to show that the matroidal locus is indeed the biggest open subset where $\AgPer$ and
$\AgVor$ are isomorphic.


Let us introduce some notations in order to describe our results in more detail.
A real $g\times n$ matrix $A\in M_{g,n}(\R)$ is called totally unimodular if every square submatrix of $A$ has determinant equal to $-1$, $0$ or $1$. A matrix $A\in M_{g,n}(\R)$ is called unimodular if there exists $h\in \GL_g(\Z)$ such that $h \cdot A$ is totally unimodular.
Given a $g\times n$ unimodular matrix $A\in M_{g,n}(\R)$ with column vectors $\{v_1,\ldots,v_n\}$, we define a rational polyhedral subcone $\sigma(A)$ of $\Ort$
as the convex hull of the rank $1$ quadratic forms $\{v_i\cdot v_i^t\}_{i=1,\ldots,n}$.
The union of the cones $\sigma(A)$, as $A$ varies among all the  unimodular matrices $A\in M_{g,n}(\R)$ of rank
at most $g$, forms a subcone of $\Ort$, denoted by $\Omat$ and called the \emph{matroidal subcone.} The collection of the cones $\{\sigma(A)\}$ is called the
\emph{matroidal decomposition} of $\Omat$ and is denoted by $\Mat$. The name matroidal comes from the fact that  unimodular matrices $A\in M_{g,n}(\R)$ of rank at most $g$
up to the natural action  of $\GL_g(\Z)$ by left multiplication are in bijection with regular matroids of rank at most $g$ (see Fact \ref{F:reg-thm}).
In particular, the $\GL_g(\Z)$-equivalence classes of cones in $\Mat$ correspond bijectively to (simple) regular matroids of rank at most $g$ (see  Lemma \ref{L:prop-mat-cones}).
Our first main result is the following (see Corollary \ref{C:main-result}).

\begin{theoremalpha}\label{T:thmA}
A cone $\sigma$ belongs to both $\Voro$ and $\Per$ if and only if $\sigma$ belongs to $\Mat$, i.e.
$$\Voro\cap \Per= \Mat.$$
\end{theoremalpha}

\noindent The proof of the above Theorem \ref{T:thmA} is divided into three parts: we begin by proving that $\Mat$ is contained in $\Voro$, then we show that $\Mat$ is contained in $\Per$ and finally we prove that $\Per\cap \Voro$ is contained in $\Mat$.

The fact that $\Mat\subseteq \Voro$ is a result of
Erdhal-Ryshkov \cite{ER}: they prove that $\Mat$ is the subset of $\Voro$ corresponding to cones whose associated
Delone subdivision is a lattice dicing (see Section \ref{S:Mat-Voro} for details).

In order to prove that $\Per\cap \Voro\subseteq \Mat$, we use the fact that $\Per$ is made of cones
whose extremal rays are generated by rank 1 quadratic forms together with a result of Erdhal-Ryshov \cite{ER} that
characterizes $\Mat$ as the collection of cones of $\Voro$ satisfying the above property.

The proof of $\Mat\subseteq \Per$ is the hardest part. To achieve that, we use Seymour's decomposition theorem
which says that any regular matroid can be obtained, via a sequence of 1-sums, 2-sums and 3-sums, from three kinds of basic
matroids: graphic, cographic and a special matroid called $R_{10}$ (see Section \ref{S:matroids} for details).
A crucial role is played by a result of Alexeev-Brunyate (see \cite[Thm. 5.6]{AB}) which, in our language, says that
if $A$ is a unimodular matrix representing a cographic matroid, then $\sigma(A)\in \Per$. The authors of loc. cit. asked in \cite[6.3]{AB} if their result could be extended from cographic matroids to regular matroids and, indeed,
Theorem \ref{T:thmA}  answers positively to their question.

In the last part of the paper we explore the consequences of Theorem \ref{T:thmA} in terms of the relationship between the toroidal compactifications of $\Ag$ that we mentioned before: $\AgPer$ and $\AgVor$.
Indeed,
 the matroidal decomposition $\Mat$ of $\Omat\subseteq \Ort$
 yields a partial compactification $\Agmat$ of $\Ag$, i.e. an irreducible variety containing $\Ag$ as an open dense subset (see Definition \ref{D:mat-locus}). Since $\Mat\subseteq \Per$ and $\Mat\subseteq \Voro$,
$\Agmat$ is an open subset of both $\AgPer$ and $\AgVor$.
Our second main result is the following (see Theorem \ref{T:compa-toro} for a more precise version).

\begin{theoremalpha}\label{T:mainB}
\noindent
\begin{enumerate}[(i)]
\item  $\Agmat$ is the biggest open subset of $\AgVor$ where the rational map $\AgVor\stackrel{\tau}{\dashrightarrow} \AgPer$ is defined and is an isomorphism.
\item $\Agmat$ is the biggest open subset of $\AgPer$ where the rational map $\AgPer \stackrel{\tau^{-1}}{\dashrightarrow} \AgVor$ is defined.
\item  The compactified Torelli maps $\tgb^P$ and $\tgb^V$ fit into the following commutative diagram
    $$\xymatrix{
    & \Agmat \ar@{^{(}->}[r] \ar@{=}[dd]& \AgVor \ar@{-->}[dd]^{\tau} \\
    \Mgb \ar[ru]^{\tgb^V} \ar[rd]_{\tgb^P} & & \\
    & \Agmat \ar@{^{(}->}[r] & \AgPer\\
    }$$
\end{enumerate}
\end{theoremalpha}

Finally we want to mention that there exists a third well-known admissible decomposition of $\Ort$, namely the central cone decomposition $\Sigma_C$ (see \cite{Koe} and \cite[Sec. (8.9)]{NamT}). The toroidal compactification
$\AgCen$ associated to
$\Sigma_C$ is known to be the normalization of the blow-up of the Satake compactification of $\Ag$ along the boundary
(see \cite{Igu}). However, the comparison of $\Sigma_C$ with $\Per$ and with $\Voro$ seems to be less obvious.
For example, it follows from \cite[Cor. 4.6]{AB}, that $\AgCen$ does not contain an open subset isomorphic to
$\Agmat$ at least if $g\geq 9$. For the same reason,  the Torelli map does not extend to a regular map from $\Mgb$ to $\AgCen$ for $g\geq 9$ (while it does for $g\leq 8$ by \cite{VIGRE}).

The structure of the paper is as follows. In Section \ref{S:admi-deco}, we first recall the definition of an admissible
decomposition of $\Ort$, and then we review the definition and the basic properties of the perfect cone decomposition (Section
\ref{S:Perfect}) and of the 2nd Voronoi decomposition (Section \ref{S:Voro}). In Section \ref{S:matroids}, we briefly review the basic concepts of matroid theory that
we will need throughout the paper, with particular emphasis on Seymour's decomposition theorem of regular matroids
(Section \ref{S:Seymour}). Section \ref{S:matroid-locus}
is devoted to the proof of Theorem \ref{T:thmA} (see Corollary \ref{C:main-result}). Section \ref{S:toroidal} starts with a brief review of the theory of toroidal compactifications of $\Ag$ and ends with a proof of Theorem \ref{T:mainB} (see Theorem \ref{T:compa-toro}).

This paper is meant to be completely self-contained, so we have tried to recall all the preliminary notions necessary to its understanding by readers with a background either on combinatorics or on algebraic geometry.

\section{Positive definite quadratic forms and admissible decompositions}\label{S:admi-deco}

We denote by $\R^{\binom{g+1}{2}}$ the vector space of quadratic forms
in $\R^g$ (identified with $g\times g$ symmetric matrices with
coefficients in $\R$) and by $\O$ the cone in $\R^{\binom{g+1}{2}}$ of positive
definite quadratic forms. The closure $\ov{\O}$ of $\O$ inside $\R^{\binom{g+1}{2}}$ is the cone
of positive semi-definite quadratic forms. We will be working with a partial closure of the cone $\O$ inside $\ov{\O}$, the so called rational closure of $\O$ (see \cite[Sec. 8]{NamT}).

\begin{defi}\label{D:rat-forms}
A positive definite quadratic form $Q$ is said to be \emph{rational} if the null space $\Null(Q)$ of $Q$
(i.e. the biggest subvector space $V$ of $\R^g$ such that $Q$ restricted to $V$ is identically zero)
admits a basis with elements in $\Q^g$.

We will denote by $\Ort$ the cone of rational positive semi-definite quadratic forms.
\end{defi}

The group $\GL_g(\Z)$ acts on the vector space $\R^{\binom{g+1}{2}}$ of quadratic forms
via the usual law $h\cdot Q:= h Q h^t$, where $h\in \GL_g(\Z)$ and $h^t$ is the transpose matrix.
Clearly the cones $\O$ and $\Ort$ are preserved by the action of $\GL_g(\Z)$.

\begin{remark}\label{rat-qua}
It is well-known (see \cite[Sec. 8]{NamT}) that a positive semi-definite
quadratic form $Q$ in $\R^g$ belongs to $\Ort$ if and only if
there exists $h\in \GL_g(\Z)$ such that
$$hQh^t=
\left(\begin{array}{cc}
Q' & 0 \\
0 &  0 \\
\end{array}\right)
$$
for some positive definite quadratic form $Q'$ in $\R^{g'}$, with $0\leq g'\leq g$.
\end{remark}

The cones $\O$ and its rational closure $\Ort$ are not polyhedral. However they can be subdivided into rational polyhedral subcones in a nice way, as in the following definition (see \cite[Lemma 8.3]{NamT} or \cite[Chap. IV.2]{FC}).

\begin{defi}\label{decompo}
An \emph{admissible decomposition} of $\Ort$ is a collection $\Sigma=\{\sigma_{\mu}\}$ of rational polyhedral cones of
$\Ort$ such that:
\begin{enumerate}[(i)]
\item If $\sigma$ is a face of $\sigma_{\mu}\in \Sigma$ then $\sigma\in \Sigma$;
\item The intersection of two cones $\sigma_{\mu}$ and $\sigma_{\nu}$ of $\Sigma$ is a face of both cones;
\item If $\sigma_{\mu}\in \Sigma$ and $h\in \GL_g(\Z)$ then $h\cdot \sigma_{\mu}
\cdot h^t\in \Sigma$.
\item $\#\{\sigma_{\mu}\in \Sigma \mod \GL_g(\Z)\}$ is finite;
\item $\cup_{\sigma_{\mu}\in \Sigma} \sigma_{\mu}=\Ort$.
\end{enumerate}
We say that two cones $\sigma_{\mu}, \sigma_{\nu}\in \Sigma$ are equivalent if they are conjugated by an element of
$\GL_g(\Z)$. We denote by $\Sigma/\GL_g(\Z)$ the finite set of equivalence classes of cones in $\Sigma$. Given a cone $\sigma_{\mu}\in \Sigma$, we denote by $[\sigma_{\mu}]$ the equivalence class
containing $\sigma_{\mu}$.
\end{defi}

A priori, there could exist infinitely many admissible decompositions of $\Ort$. However,
as far as we know, only three admissible decompositions are known for every integer $g$
(see \cite[Chap. 8]{NamT} and the references there), namely:
\begin{enumerate}[(i)]
\item The perfect cone decomposition (also known as the first Voronoi
decomposition), which was first introduced in \cite{Vor};
\item The 2nd Voronoi decomposition (also known as the L-type decomposition), which was first introduced in \cite{Vor};
\item The central cone decomposition, which was introduced in \cite{Koe}.
\end{enumerate}
Each of them plays a significant (and different) role in the theory of the
toroidal compactifications of the moduli space of principally
polarized abelian varieties (see \cite{Igu}, \cite{alex1}, \cite{SB}).
We will come back to this later on.

\begin{example}
If $g=2$ then all the above three admissible decompositions coincide.
In Figure \ref{VorFig} we illustrate a section of the
$3$-dimensional cone $\Omega_2^{\rm rt}$, where we represent just some of the infinite
cones of the known admissible decompositions. Note that, for $g=2$, there is only one
$\GL_g(\Z)$-equivalence class of maximal dimensional cones, namely
the principal cone $\prin^0$ (see Example \ref{small-dim}).

\begin{figure}[h]
\begin{center}
\scalebox{.5}{\epsfig{file=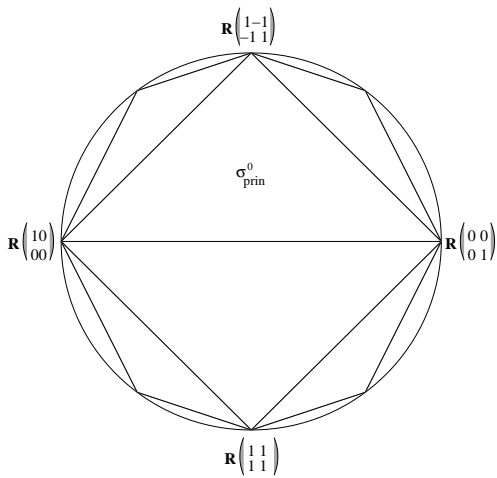}}
\end{center}
\caption{A section of $\Omega_2^{\rm rt}$ and its admissible
decomposition.}
\label{VorFig}
\end{figure}
\end{example}

In this paper, we will be interested in comparing the perfect cone decomposition with the 2nd Voronoi decomposition;
so we start by recalling briefly their definitions.

\subsection{The perfect cone decomposition $\Per$}\label{S:Perfect}
\indent

In this subsection, we review the definition and the main properties of the perfect cone decomposition
(see \cite{Vor} for more details and proofs, or \cite[Sec. (8.8)]{NamT} for a summary).

Consider the function $\mu:\O\to \R_{> 0}$ defined by
$$\mu(Q):=\min_{\xi\in \Z^g\setminus \{0\}} Q(\xi).$$
It can be checked that, for any $Q\in \O$, the set
$$M(Q):=\{\xi\in \Z^g\: : \: Q(\xi)=\mu(Q) \} $$
is finite and non-empty.
For any $\xi\in M(Q)$, consider the rank one quadratic form $\xi\cdot \xi^t\in \Ort$.
We denote by $\sigma[Q]$ the rational polyhedral subcone of $\Ort$ given by the convex hull of the rank one forms
obtained from elements of $M(Q)$, i.e.
\begin{equation}\label{D:perf-cones}
\sigma[Q]:=\R_{\geq 0}\langle \xi\cdot \xi^t\rangle_{\xi\in M(Q)}.
\end{equation}
One of the main results of \cite{Vor} is the following

\begin{fact}[Voronoi]
\label{F:main-Per}
The set of cones
$$\Per:= \{ \sigma[Q] \: :\:  Q\in \O \} $$
yields an admissible decomposition of $\Ort$, known as the {\bf perfect cone decomposition}.
\end{fact}

The quadratic forms $Q$ such that $\sigma[Q]$ has maximal dimension $\binom{g+1}{2}$ are called \emph{perfect},
hence the name of this admissible decomposition. The interested reader is referred to  \cite{Mar} for more details on perfect forms.

\begin{remark}\label{R:Per-simpli}
\noindent \begin{enumerate}[(i)]
\item The cones $\sigma[Q] \in \Per$ need not be simplicial for $g\geq 4$ (see \cite[p. 93]{NamT}).
\item It follows easily from the definition that the extremal rays of
the cones $\tau\in \Per$ are generated by quadratic forms of rank one.
Moreover, it is easily checked that the cone $\langle Q\rangle$ generated by any rank-1 quadratic forms $Q\in \Ort$ belongs to $\Sigma_P$. In particular, from the properties of an admissible decomposition (see Definition \ref{decompo}), it follows that if $Q\in \Ort$ is a rank-1 quadratic form belonging to a cone $\tau\in \Sigma_P$, then $\langle Q\rangle$ is an extremal ray of $\tau$.
\end{enumerate}
\end{remark}

\begin{example}\label{E:Per-g2}
 Let us compute $\Per$ in the case $g=2$ (compare with Figure \ref{VorFig}).
Let
$R_{12} = \left(\begin{matrix}1&-1\\-1&1\end{matrix}\right)$,
$R_{13} = \left(\begin{matrix}1&0\\0&0\end{matrix}\right)$,
$R_{23} = \left(\begin{matrix}0&0\\0&1\end{matrix}\right).$
Then, up to $\GL_g(\Z)$-equivalence, an easy computation shows that the unique cones in $\Per$ are
\begin{align*}
 &\sigma\left[\left(\begin{matrix}1&1/2\\1/2&1\end{matrix}\right)\right] = \R_{\geq 0} \langle R_{12}, R_{13}, R_{23} \rangle=\left\{\left(\begin{matrix}a+c&-c\\-c&b+c\end{matrix}\right) \: : \: a, b, c \geq 0\right\} , \\
 &\sigma\left[\left(\begin{matrix}1&\lambda\\\lambda&1\end{matrix}\right)\right] = \R_{\geq 0} \langle R_{13}, R_{23} \rangle=\left\{\left(\begin{matrix}a&0\\0&b\end{matrix}\right) \: : \: a, b \geq 0\right\} \text{ for any } -1/2<\lambda <1/2, \\
 &\sigma\left[\left(\begin{matrix}1&\lambda\\\lambda&\mu\end{matrix}\right)\right] = \R_{\geq 0} \langle R_{13} \rangle=
 \left\{\left(\begin{matrix}a&0\\0&0\end{matrix}\right) \: : \: a \geq 0\right\} \text{ for any } \mu>\max\{1, \lambda^2, \pm 2\lambda\}, \\
 &\sigma\left[\left(\begin{matrix}\nu&\lambda\\\lambda&\mu\end{matrix}\right)\right] = \{0\} \text{ for any } \mu, \nu>1, \mu\nu>\lambda^2, \mu+\nu>1\pm 2\lambda.
\end{align*}

\end{example}

\subsection{The 2nd Voronoi decomposition $\Voro$}\label{S:Voro}
\indent

In this subsection, we review the definition and main properties of the 2nd Voronoi admissible decomposition
(see \cite{Vor}, \cite[Chap. 9(A)]{NamT} or \cite[Chap. 2]{Val} for more details and proofs).

The 2nd Voronoi decomposition is based on the Delone subdivision $\Del(Q)$ associated to a quadratic form $Q\in \Ort$.

\begin{defi}
  Given $Q \in \Ort$, consider the map $l_Q : \Z^g \to \Z^g \times \R$ sending $x \in \Z^g$ to $(x,Q(x))$. View the image of $l_Q$ as an infinite set of points in $\R^{g+1}$, one above each point in $\Z^g$, and consider the convex hull of these points. The lower faces of the convex hull
 can now be projected to $\R^g$ by the map $\pi:\R^{g+1}\to \R^g$ that forgets the last coordinate. This produces an infinite $\Z^g$-periodic polyhedral subdivision of $\R^g$, called the {\bf Delone subdivision} of $Q$ and denoted $\Del(Q)$.
\end{defi}

It can be checked that if $Q$ has rank $g'$ with $0\leq g'\leq g$ then $\Del(Q)$ is a subdivision consisting of polyhedra
such that the maximal linear subspace contained in them has dimension $g-g'$. In particular, $Q$ is positive definite
if and only if $\Del(Q)$ is made of polytopes, i.e. bounded polyhedra.

Now, we group together quadratic forms in $\Ort$ according to the Delone subdivisions that they yield.

\begin{defi}\label{D:open-secondary}
  Given a Delone subdivision $D$ (induced by some $Q_0\in \Ort$), let
  \[ \sigma_D^0 = \{ Q \in \Ort : \Del(Q) = D \}. \]
\end{defi}

It can be checked  that the set $\sigma_D^0$ is a relatively open (i.e. open in its linear span) rational polyhedral cone in $\Ort$.
Let $\sigma_D$ denote the Euclidean closure of $\sigma_D^0$ in $\R^{\binom{g+1}{2}}$, so $\sigma_D$ is a closed rational polyhedral cone and $\sigma_D^0$ is its relative interior. We call $\sigma_D$ the {\bf secondary cone} of $D$.
The reason for this terminology is due to the fact that Alexeev has shown in \cite{alex1} that the 2nd Voronoi
decomposition is an infinite periodic analogue of the secondary fan of Gelfand-Kapranov-Zelevinsky (see \cite{GKZ}).

Now, the action of the group $GL_g(\Z)$ on $\R^g$ induces an action of $GL_g(\Z)$ on the set of Delone subdivisions: given a Delone subdivision $D$ and an element $h\in \GL_g(\Z)$, denote by $h\cdot D$ the Delone subdivision given by the action of $h$ on $D$. Moreover, $\GL_g(\Z)$ acts naturally on the set
of secondary cones $\{\sigma_D:D \text{ is a Delone subdivision of } \R^g\}$ in such a way that
$$h\cdot \sigma_D:=\{hQh^t \: : \: Q\in \sigma_D \}=\sigma_{h\cdot D}.$$
Another of the main results of \cite{Vor} is the following

\begin{fact}[Voronoi]
\label{F:main-Vor}
The set of secondary cones
$$\Voro:= \{ \sigma_D : D \text{ is a Delone subdivision of } \R^g \} $$
yields an admissible decomposition of $\Ort$, known as the {\bf 2nd Voronoi decomposition}.
\end{fact}

The cones of $\Voro$ having maximal dimension $\binom{g+1}{2}$ are those of the form $\sigma_D$ for $D$ a Delone subdivision which is a triangulation, i.e. such that $D$ consists only of 
simplices (see \cite[Sec. 2.4]{Val}).


     
     

The following remark should be compared with Remark \ref{R:Per-simpli}.

\begin{remark}\label{R:Voro-nonsimpli}
\noindent
\begin{enumerate}[(i)]
\item The cones $\sigma_D\in \Voro$ need not be simplicial for $g\geq 5$ (see \cite{BT} and \cite{EG}).
\item If $Q\in \Ort$ belongs to a one dimensional cone $\sigma_D\in \Voro$ (or in other words, if $Q$ generates an extremal  ray of some cone of $\Voro$) then $Q$ is said \emph{rigid}.
  Rank-1 quadratic forms $Q\in \Ort$ are easily seen to be rigid. In particular if $Q\in \Ort$ is a rank-1 quadratic form belonging to a cone $\tau\in \Sigma_V$, then the cone $\langle Q\rangle$ generated by $Q$ is an extremal ray of $\tau$.

  However,  rigid forms need not to be of rank one for $g\geq 4$ (see \cite{BG}, \cite{DG0} and \cite{DV}).
\end{enumerate}
\end{remark}

There is another way of describing the 2nd Voronoi decomposition via the {\bf Dirichlet-Voronoi
polytope} $\Vor(Q)$ associated to a quadratic form $Q\in \Ort$ (see \cite[Chap. 9(A)]{NamT} or \cite[Chap. 3]{Val} for more details). Given a positive definite quadratic form $Q\in \O$, we define $\Vor(Q)$ as
\begin{equation}\label{def-Vor}
\Vor(Q):=\{x\in \R^g \: :\: Q(x)\leq Q(v-x)
\text{ for all } v\in \Z^g\}.
\end{equation}
More generally, if $Q=h
\left(\begin{array}{cc}
Q' & 0 \\
0 &  0 \\
\end{array}\right)
h^t$ for some $h\in \GL_g(\Z)$ and some po\-si\-ti\-ve definite quadratic form $Q'$ in $\R^{g'}$, $0\leq g'\leq g$ (see Remark \ref{rat-qua}), then $\Vor(Q):=h^{-1}\Vor(Q')(h^{-1})^t\subset
h^{-1}\R^{g'}(h^{-1})^t$.
In particular, the smallest linear subspace $\langle \Vor(Q)\rangle$ containing $\Vor(Q)$
has dimension equal to the rank of $Q$. The integral translates of $\Vor(Q)$
$$\{\Vor(Q)+v\}_{v\in \langle \Vor(Q)\rangle \cap \Z^g}$$
form a face to face tiling (in the sense of \cite{She} and \cite{McM}) of the vector space $\langle \Vor(Q)\rangle$ which is dual to the Delone subdivision
$\Del(Q)$ (see \cite[Chap. 9(A)]{NamT} or \cite[Sec. 3.3]{Val} for details). From this fact, it follows easily that, for a Delone subdivision $D=\Del(Q_0)$ induced by $Q_0\in \Ort$, the cone $\sigma_D^0$ of Definition \ref{D:open-secondary} is also equal to the set of $Q\in \Ort$ such that $\Vor(Q)$ is normally equivalent to $\Vor(Q_0)$, i.e. such that
$\Vor(Q)$ and $\Vor(Q_0)$ have the same normal fan.

\begin{example}
\label{e:Atr2}
  Let us compute $\Voro$ in the case $g=2$ (compare with Figure \ref{VorFig} and with Example \ref{E:Per-g2}).  Combining the taxonomies in \cite[Sec. 4.1, Sec. 4.2]{Val}, we may choose four representatives $D_1, D_2, D_3, D_4$ for $\GL_g(\Z)$-orbits of Delone subdivisions as in Figure \ref{f:subdiv}, where we have depicted the part of the Delone subdivision that fits inside the unit cube in $\R^2$.

\begin{figure}[h]%
\includegraphics[width=3in]{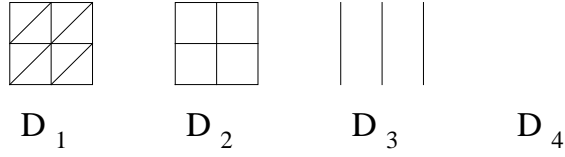}%
\caption{Delone subdivisions for $g=2$ (up to $\GL_g(\Z)$-equivalence).}%
\label{f:subdiv}%
\end{figure}
\noindent We can describe the corresponding secondary cones as follows. Let
$R_{12} = \left(\begin{matrix}1&-1\\-1&1\end{matrix}\right)$,
$R_{13} = \left(\begin{matrix}1&0\\0&0\end{matrix}\right)$,
$R_{23} = \left(\begin{matrix}0&0\\0&1\end{matrix}\right)$ as in Example \ref{E:Per-g2}.
Then
\begin{align*}
 \sigma_{D_1} &= \R_{\geq 0} \langle R_{12}, R_{13}, R_{23} \rangle=\left\{\left(\begin{matrix}a+c&-c\\-c&b+c\end{matrix}\right) \: : \: a, b, c \geq 0\right\} , \\
 \sigma_{D_2} &= \R_{\geq 0} \langle R_{13}, R_{23} \rangle=\left\{\left(\begin{matrix}a&0\\0&b\end{matrix}\right) \: : \: a, b \geq 0\right\} , \\
 \sigma_{D_3} &= \R_{\geq 0} \langle R_{13} \rangle=
 \left\{\left(\begin{matrix}a&0\\0&0\end{matrix}\right) \: : \: a \geq 0\right\} , \\
 \sigma_{D_4} &= \{0\}.
\end{align*}
\end{example}

\section{Matroids}\label{S:matroids}

The aim of this section is to recall the basic notions and results of (unoriented) matroid theory that we will
need in the sequel. We follow mostly the terminology and notations
of \cite{Oxl}.

\subsection{Basic definitions}


There are several ways of defining a matroid (see \cite[Chap. 1]{Oxl}).
We will use the definition in terms of bases (see \cite[Sect. 1.2]{Oxl}).

\begin{defi}\label{matroid}
A {\it matroid} $M$ is a pair $(E(M), \BB(M))$ where $E(M)$ is a finite set, called the {\it ground set}, and $\BB(M)$ is a collection of subsets of $E(M)$, called {\it bases}
of $M$, satisfying the following two conditions:
\begin{enumerate}[(i)]
\item $\BB(M)\neq \emptyset$;
\item If $B_1,B_2\in \BB(M)$ and $x\in B_1\setminus B_2$, then there exists
an element $y\in B_2\setminus B_1$ such that $(B_1\setminus \{x\})\cup \{y\}
\in \BB(M)$.
\end{enumerate}

Given a matroid $M=(E(M), \BB(M))$, we define:
\begin{enumerate}[(a)]
\item The set of independent elements
$$\II(M):=\{I\subset E(M)\,:\, I\subset B \text{ for some } B\in \BB(M)\};$$
\item The set of dependent elements
$$\DD(M):=\{D\subset E(M)\,:\,  D \not\in \II(M)\};$$
\item The set of circuits
$$\CC(M):=\{C\in \DD(M)\,:\, C \text{ is minimal among the elements of } \DD(M)\}.$$
\end{enumerate}
\end{defi}

It can be derived from the above axioms, that all the bases of $M$ have the
same cardinality, which is called the {\it rank} of $M$ and is denoted by $r(M)$.

Observe that each of the above sets $\BB(M)$, $\II(M)$, $\DD(M)$, $\CC(M)$
determines all the others. Indeed, it is possible to define a matroid $M$ in terms of the ground set $E(M)$ and each of the above sets, subject to suitable
axioms (see \cite[Sec. 1.1, 1.2]{Oxl}).

The above terminology comes from the following basic example of
matroids.

\begin{example}\label{rep-mat}
Let $F$ be a field and $A$ an $r\times n$ matrix of rank $r$ over $F$.
Consider the columns of $A$ as elements of the vector space $F^r$,
and call them $\{v_1, \ldots, v_n\}$.
The {\bf vector matroid} of $A$, denoted by $M[A]$, is the matroid
whose ground set is $E(M[A]):=\{v_1,\ldots,v_n\}$ and whose bases are the
subsets of $E(M[A])$ consisting of vectors that form a base of
$F^r$. It follows easily that $\II(M[A])$ is formed by the subsets
of independent vectors of $E(M[A])$; $\DD(M[A])$
is formed by the subsets of dependent vectors and $\CC(M[A])$ is formed
by the minimal subsets of dependent vectors.
\end{example}

The matroids we will deal with in this paper are simple and regular.
Let us begin by recalling the definition of a simple matroid
(see \cite[Pag. 13, Pag. 52]{Oxl}).

\begin{defi}\label{simple-mat}
Let $M$ be a matroid. An element $e\in E(M)$ is called a {\it loop} if
$\{e\}\in \CC(M)$. Two distinct elements $f_1, f_2\in E(M)$ are called
{\it parallel} if $\{f_1, f_2\}\in \CC(M)$; a parallel class of $M$
is a maximal subset $X\subset E(M)$ with the property that
all the elements of $X$ are not loops and they are pairwise
parallel.

$M$ is called {\bf simple} if it has no loops and all the parallel classes
have cardinality one.
\end{defi}




\begin{example}
A vector matroid $M[A]$ is simple if and only if
$A$ has no zero columns nor pairs of proportional columns.
 In this case, we say that the matrix $A$ is \emph{simple}.

\end{example}

We now recall the definition of regular matroids.

\begin{defi}
A matroid $M$ is said to be {\it representable} over a field $F$ if it is isomorphic to the vector matroid
of a matrix $A$ with coefficients in $F$. A matroid $M$ is said to be
{\bf regular} if it is representable over any field $F$.
\end{defi}

Regular matroids are closely related to totally unimodular matrices or, more generally, to unimodular
matrices.

\begin{defi}\label{D:unimod}
\noindent
\begin{enumerate}
\item A real matrix $A\in M_{g,n}(\R)$ is said to be {\it totally unimodular}
if every square submatrix has determinant equal to $-1$, $0$ or $1$.
A matrix $A\in M_{g,n}(\Z)$ is said to be {\it unimodular} if there exists a matrix
$h\in {\rm GL}_g(\Z)$ such that $h A$ is totally unimodular.

\item We say that two  unimodular matrices $A, B\in M_{g,n}(\R)$
are {\it equivalent} if $A=hBY$ where $h\in {\rm GL}_g(\Z)$ and $Y\in {\rm GL}_n(\Z)$
is a signed permutation matrix.
\end{enumerate}
\end{defi}

\begin{fact}\label{F:reg-thm}
\begin{enumerate}[(i)]
\item A matroid $M$ of rank $r$ is regular if and only if $M=M[A]$ for a
unimodular (equivalently, totally unimodular) matrix $A\in M_{g,n}(\R)$ of rank $r$, where $n=\#E(M)$
and $g$ is a
natural number such that $g\geq r$.
\item Given two unimodular matrices $A, B\in M_{g, n}(\R)$,
we have that $M[A]=M[B]$ if and only if $A$ and $B$ are equivalent.
\end{enumerate}
\end{fact}
\begin{proof}
Part $(i)$ is proved in  \cite[Thm. 6.3.3]{Oxl}. Part $(ii)$ follows easily from
\cite[Prop. 6.3.13, Cor. 10.1.4]{Oxl}, taking into account that $\R$
does not have non-trivial automorphisms.
\end{proof}


\subsection{Graphic and cographic matroids}


There are two matroids that can be naturally associated to a graph: a graphic matroid and a cographic matroid.
We will briefly review these constructions since they will play a key role in the sequel.

Recall first the following basic concepts of graph theory (we follow mostly the terminology of \cite{Die}).
Given a graph $\Gamma$ (which we assume always to be finite, connected and possibly with loops or multiple edges),
denote by $V(\Gamma)$ the set of vertices of $\Gamma$ and by $E(\Gamma)$ the set of edges of $\Gamma$.
Given a set $S\subseteq E(\Gamma)$, the subgraph of $\Gamma$ induced by $S$ is the subgraph whose edges are the edges in $S$ and whose vertices are the vertices of $\Gamma$ which are endpoints of edges in $S$. Given a set $W\subseteq E(\Gamma)$, the subgraph of $\Gamma$ induced by $W$ is the graph whose vertices are the vertices in $W$ and whose edges are the edges of $\Gamma$ whose both endpoints are vertices in $W$.
The \textit{valence} of a vertex $v$, denoted by $\val(v)$, is defined as the number of edges incident to $v$, with the usual convention that a loop around a vertex $v$ is counted twice in the valence of $v$. A graph $\Gamma$ is $k$\textit{-regular} if $\val(v)=k$ for every $v\in V(\Gamma)$. A graph $\Gamma$ is \textit{simple} if $\Gamma$ has no loops nor multiple edges. A graph $\Gamma$ is \textit{$k$-edge connected} (for some $k\geq 2$) if and only if $\Gamma$ cannot be disconnected by deleting $1 \leq s\leq k-1$ edges.

\begin{defi}\label{cycles&bonds}
A \textit{circuit} of $\Gamma$ is a subset $S\subseteq E(\Gamma)$ such that the subgraph of $\Gamma$ induced by $S$
is $2$-regular. A \textit{cycle} is a disjoint union of circuits.\\
If $\{V_1,V_2\}$ is a partition of $V(\Gamma)$, the set $E(V_1,V_2)$ of all the edges of $\Gamma$ with one end in $V_1$ and the other end in $V_2$ is called a \textit{cut}; a \textit{bond} is a minimal cut, or equivalently, a cut $E(\Gamma_1,\Gamma_2)$
such that the graphs $\Gamma_1$ and $\Gamma_2$ induced by $V_1$ and $V_2$, respectively, are connected.
\end{defi}

\begin{defi}\label{grap&cogr-mat}
The {\bf graphic matroid} (or {\it cycle matroid}) of $\Gamma$ is the matroid
$M(\Gamma)$ whose ground set is $E(\Gamma)$
and whose circuits are the circuits of $\Gamma$. The {\bf cographic matroid} (or {\it bond matroid}) of $\Gamma$ is the matroid
$M^*(\Gamma)$ whose ground set is $E(\Gamma)$ and whose circuits
are the bonds of $\Gamma$.
\end{defi}

We summarize the well-known properties of the graphic and cographic matroids that we will need later on in the following

\begin{fact}\label{F:graph-cograph}
Let $\Gamma$ be a (finite connected) graph. Then:
\begin{enumerate}[(i)]
\item $M(\Gamma)$ and $M^*(\Gamma)$ are regular.
\item $M(\Gamma)$ is simple if and only if $\Gamma$ is simple.
$M^*(\Gamma)$ is simple if and only if $\Gamma$ is $3$-edge connected, i.e. $\Gamma$ cannot be disconnected by deleting one or two edges.
\item The rank of $M(\Gamma)$ is the cogenus $g^*(\Gamma):=|V(\Gamma)|-1$ of $\Gamma$. The rank of $M^*(\Gamma)$ is the genus $g(\Gamma):=|E(\Gamma)|-|V(\Gamma)|+1$ of $\Gamma$.
\end{enumerate}
\end{fact}
\begin{proof}
Part (i) follows from \cite[Prop. 5.1.3, Prop. 2.2.22]{Oxl}.

Part (ii) for $M(\Gamma)$ follows from \cite[Pag. 52]{Oxl} and for $M^*(\Gamma)$ follows from \cite[Prop. 2.3.14(ii)]{BMV}.

Part (iii) follows from \cite[Pag. 26]{Oxl} and \cite[Formula 2.1.8]{Oxl}.
\end{proof}

\begin{example}\label{E:complete}
Let $K_{g+1}$ be the complete simple graph on $g+1\geq 2$ vertices, i.e. the graph with vertex set $\{v_1,\ldots,v_{g+1}\}$ and edge set $\{e_{ij}\: : \:1 \leq i<j\leq g+1\}$, where $e_{ij}$ is an edge joining $v_i$ and $v_j$.
It is easy to check (see \cite[Prop. 5.1.2, Prop. 5.1.3]{Oxl}) that $M(K_{g+1})$ is a simple regular matroid of rank
$g$ which can be obtained as the vector matroid associated to the simple totally unimodular matrix $A(K_{g+1})\in
M_{g,\binom{g+1}{2}}(\Z)$ whose column vectors are the vectors $\{\vec e_i\: : \: 1\leq i\leq g \}$ and
 $\{\vec e_i-\vec e_j\: : \: 1\leq i<j\leq g\}$ of $\R^g$, where $\{\vec e_1,\ldots,\vec e_g\}$ denotes the canonical bases of $\R^g$.
\end{example}





\subsection{The matroid $R_{10}$}


Another matroid that will play a key role in the sequel is the matroid  $R_{10}$ introduced in \cite[p. 328]{Sey}.

\begin{defi}\label{D:R_10}
We denote by $R_{10}$ the vector matroid associated to the totally unimodular simple matrix
$$A_{10}:=\left(\begin{matrix}
1&0&0&0&0&-1&1&0&0&1\\
0&1&0&0&0&1&-1&1&0&0\\
0&0&1&0&0&0&1&-1&1&0\\
0&0&0&1&0&0&0&1&-1&1\\
0&0&0&0&1&1&0&0&1&-1\\
\end{matrix}\right).
$$
\end{defi}
It is easy to see that $R_{10}$ is a simple regular matroid of rank $5$.

We mention that, quite recently, the matroid $R_{10}$ has made a striking appearance in algebraic geometry: Gwena has shown in \cite{Gwe} that $R_{10}$ is related to the degenerations
of the intermediate Jacobians associated to a family of cubic threefolds degenerating to the Segre's
cubic in $\mathbb{P}^4$.


\subsection{Seymour's decomposition theorem}\label{S:Seymour}


Here we review Seymour's decomposition theorem (see \cite{Sey}) which says that regular matroids can be obtained starting from graphic matroids, cographic matroids and the matroid $R_{10}$ via
simple operations called $1$-sum, $2$-sum and $3$-sum. However, since we want a Seymour's decomposition theorem inside the category of simple regular matroids (while Seymour's original formulation works only in the category of
all regular matroids, possibly non simple), we prefer to adopt the slightly modified
constructions of Danilov and Grishukhin (see \cite{DG}) \footnote{Note however that in \cite{DG} the above modified operations are called, respectively, $0$-sum, $1$-sum and $2$-sum (with a shift in the
enumeration!);  however we will keep the original terminology of Seymour to avoid possible confusions.}.

Following \cite[p. 413]{DG}, we will give the definitions of $1$-sum, $2$-sum and $3$-sum of simple regular matroids in terms of
representations as vector matroids of simple totally unimodular matrices.


\begin{defi}\label{D:sum-matroids}
Let $M_1$, $M_2$ and $M$ be three simple regular matroids.
\begin{enumerate}[(i)]
\item We say that $M$ is the {\bf $1$-sum} of $M_1$ and $M_2$, and we write $M=M_1\oplus_1 M_2$, if we can write $M_1=M[A_1]$, $M_2=M[A_2]$ and $M=M[A]$
for some simple totally unimodular matrices $A$, $A_1$ and $A_2$ such that
$$A=\left(\begin{matrix}
A_1 & 0 \\
0 & A_2\\
\end{matrix}\right). $$
\item We say that $M$ is the {\bf $2$-sum} of $M_1$ and $M_2$, and we write $M=M_1\oplus_2 M_2$, if we can write $M_1=M[A_1]$, $M_2=M[A_2]$ and $M=M[A]$
for some simple totally unimodular matrices $A$, $A_1$ and $A_2$ such that
$$A_1=\left(\begin{matrix}
B & 0 \\
b^t & 1\\
\end{matrix}\right),
\hspace{0,5cm}
A_2=\left(\begin{matrix}
c^t & 1 \\
C & 0\\
\end{matrix}\right),
\hspace{0,5cm}
A=\left(\begin{matrix}
B & 0 & 0 \\
b^t & c^t & 1\\
0 & C & 0
\end{matrix}\right),
$$
where $B, C$ are matrices and $b, c$ are vectors.
\item We say that $M$ is the {\bf $3$-sum} of $M_1$ and $M_2$, and we write $M=M_1\oplus_3 M_2$, if we can write $M_1=M[A_1]$, $M_2=M[A_2]$ and $M=M[A]$
for some simple totally unimodular matrices $A$, $A_1$ and $A_2$ such that
$$A_1=\left(\begin{matrix}
B & 0 & 0 & 0 \\
b_1^t & 1 & 0& 1\\
b_2^t & 0 & 1 & 1
\end{matrix}\right),
\hspace{0,3cm}
A_2=\left(\begin{matrix}
c_1^t & 1 & 0 & 1\\
c_2^t & 0 & 1 & 1\\
C & 0 & 0 & 0 \\
\end{matrix}\right),
\hspace{0,3cm}
A=\left(\begin{matrix}
B & 0 & 0 & 0 & 0\\
b_1^t & c_1^t & 1 & 0 & 1\\
b_2^t & c_2^t & 0 & 1 & 1\\
0 & C & 0 & 0 & 0
\end{matrix}\right),
$$
where $B, C$ are matrices and $b_1, b_2, c_1, c_2$ are vectors.
\end{enumerate}
\end{defi}

Some remarks are in order.

\begin{remark}\label{R:sum}
\noindent
\begin{enumerate}[(i)]
\item The difference between the above definition (taken from \cite{DG}) and the original definition of Seymour (\cite[Sec. 2]{Sey})
is the following: in (ii) Seymour drops the last column of $A$; in (iii) he drops the last three columns of $A$.
\item In each of the above operations (i), (ii) or (iii),  $A_1$ and $A_2$ are totally unimodular if and only if  $A$ is totally unimodular. The if direction is clear since $A_1$ and $A_2$ are submatrix
of $A$. The only if direction is proved in \cite{Bry}.
\item It is immediate to check that, in each of the above operations (i), (ii) and (iii), if $A$ is simple then $A_1$ and $A_2$ must be simple as well.
Conversely, if we assume that $A_1$ and $A_2$ are simple and totally unimodular then we get that $A$ is simple as well.
This is clear in the operation (i). In the operations (ii) and (iii), it follows from the fact that if $A_1$ (resp. $A_2$) is simple and totally unimodular then $B$ (resp. $C$) cannot have zero column vectors
since $(1)$ cannot be a proper submatrix of a simple totally unimodular matrix of rank $1$ and, similarly, $\displaystyle \left(\begin{matrix}
1 & 0 & 1\\
0 & 1 & 1\\
\end{matrix}\right)$
cannot be a proper submatrix of a simple totally unimodular matrix of rank $2$.
\end{enumerate}
\end{remark}

We can now state the main Theorem of \cite{Sey} (see also \cite{DG}) as follows:

\begin{fact}[Seymour's decomposition theorem]\label{F:Seymour}
Every simple regular matroid can be obtained by means of $1$-sum, $2$-sum and $3$-sum starting from simple graphic matroids, simple cographic matroids and $R_{10}$.
\end{fact}


\section{The matroidal subcone $\Omat$ and its matroidal decomposition $\Mat$}\label{S:matroid-locus}

The aim of this section is to introduce and study a $\GL_g(\Z)$-invariant closed subcone of the cone $\Ort$ of rational positive semi-definite quadratic forms on $\R^g$, called the matroidal subcone and denoted by $\Omat$, and a natural admissible decomposition of it, which we call the matroidal decomposition and we denote by $\Mat$.

\begin{defi}\label{D:cone-mat}
Let $A\in M_{g,n}(\Z)$ be a simple unimodular matrix (for some $g$ and $n$). Denote its column vectors by
$\{v_1,\ldots, v_n\}\subset \R^g$. Define the closed rational polyhedral cone $\sigma(A)\subset \Ort$ as
$$\sigma(A):=\R_{\geq 0}\langle v_1v_1^t,\ldots, v_nv_n^t\rangle,$$
and denote by $\sigma(A)^0$ its relative interior.
The {\bf matroidal subcone} $\Omat$ of $\Ort$ is defined as
$$\Omat:=\bigcup_A \sigma(A) \subseteq \Ort,
$$
where the union runs over all the matrices $A\in M_{g, n}(\Z)$ as above (for some $n$). The
{\bf matroidal decomposition} of $\Omat$ is the collection $\Mat=\{\sigma(A)\}$, where $A\in M_{g,n}(\Z)$ varies among all the matrices as above (for some $n$).
\end{defi}
Note that the cone $\sigma(A)$ does not depend on the order of the columns of $A$, i.e. if $A=BY$ where
$Y\in {\rm GL}_n(\Z)$ is a signed permutation matrix then $\sigma(A)=\sigma(B)$.

In the following lemma, we collect the main properties of the cones $\sigma(A)$.

\begin{lemma}\label{L:prop-mat-cones}
Let $A, B\in M_{g, n}(\Z)$ be two simple  unimodular matrices. Denote by $\{v_1,\ldots,v_n\}$ the column vectors of $A$ and by $\{w_1\,\ldots, w_n\}$ the column vectors of $B$.
\begin{enumerate}[(i)]
\item The cone $\sigma(A)$ is simplicial and every face is of the form $\sigma(A\setminus I)$ for $I\subset \{1,\ldots, n\}$, where $A\setminus I$ is the matrix obtained from $A$ by deleting the columns corresponding to $I$.
\item $\sigma(A)$ is $\GL_g(\Z)$-equivalent to $\sigma(B)$ if and only if $A$ and $B$ are equivalent.
More precisely, if $A=hBY$ where $h\in {\rm GL}_g(\Z)$ and $Y\in {\rm GL}_n(\Z)$ is a signed permutation matrix, then
$\sigma(A)=h\sigma(B) h^t$.

In particular, the $\GL_g(\Z)$-equivalence classes of cones in $\Mat$ correspond bijectively to simple regular matroids of rank at most $g$. We will denote by $\sigma(M)$ the equivalence class corresponding to such a matroid $M$.
\end{enumerate}
\end{lemma}
\begin{proof}
The cone $\sigma(A)$ is the same as the cone constructed in \cite[Construction 4.4.2]{BMV}.
Therefore, (i) follows from \cite[Thm 4.4.4(iii)]{BMV}, while (ii) follows from \cite[Thm 4.4.4(ii)]{BMV} and
Fact \ref{F:reg-thm}.
\end{proof}

From the above lemma, we get that $\Mat$ forms an admissible decomposition of $\Omat$ (compare with
Definition \ref{decompo}).

\begin{cor}
The collection $\Mat=\{\sigma(A)\}$ is an \emph{admissible} decomposition of $\Omat$, i.e.
\begin{enumerate}[(i)]
\item If $\sigma$ is a face of $\sigma(A)\in \Mat$ then $\sigma\in \Mat$;
\item The intersection of two cones $\sigma(A)$ and $\sigma(B)$ of $\Mat$ is a face of both cones;
\item If $\sigma(A)\in \Mat$ and $h\in \GL_g(\Z)$ then $h\cdot \sigma(A)
\cdot h^t\in \Mat$;
\item $\#\{\sigma(A)\in \Mat \mod \GL_g(\Z)\}$ is finite;
\item $\bigcup_{\sigma(A)\in \Mat} \sigma(A)=\Omat$.
\end{enumerate}
\end{cor}

\subsection{$\Mat$ is contained in $\Voro$.}\label{S:Mat-Voro}

In this subsection, we are going to recall the well-known result of Erdhal-Ryshkov (\cite{ER})
according to which every cone of $\Mat$ is a cone of $\Voro$. A key role is played by the concept of lattice
dicing as introduced in \cite[Sec. 2]{ER}. However, we will need a slight generalization of the definition of loc. cit. in order to be able to deal with the cones $\sigma(A)\in \Mat$ such that $A$ has rank smaller than $g$.

\begin{defi}\label{D:latt-dicing}
A {\bf generalized lattice dicing} $\D$ of $\R^g$ (with respect to the standard lattice $\Z^g$) is
a $\Z^g$-periodic polyhedral subdivision of $\R^g$ whose polyhedra are cut out by the affine hyperplanes
$H_i+v:=\{x\in \R^g\: : \: x-v\in H_i\}$, where $v\in \Z^g$ and $\{H_1,\ldots,H_n\}$ is a (possibly empty)
collection of distinct central hyperplanes on $\R^g$ such that
\begin{enumerate}[(i)]
\item If we denote by $w_i$ a non-zero vector normal to the hyperplane $H_i$ (for $1\leq i\leq n$), then
the vector space $V_{\D}:=\langle w_1, \ldots, w_n\rangle \subseteq \R^g$ is defined over $\Q$, i.e. $V_{\D}$ admits a basis of elements of $\Q^g$.
\item If there exists a subset $I\subseteq \{1,\ldots,n\}$ and a collection of vectors $\{v_i\}_{i\in I}\subset \Z^g$ such that the intersection
$$V_{\D}\bigcap_i \{H_i+v_i\}$$
consists of one point $x$ (in this case, we say that $x$ is a \emph{vertex} of $\D$), then $x\in V_{\D}\cap \Z^g$.
\item For any $x\in V_{\D}\cap \Z^g$ and any $H_i$, there exists a unique $v\in V_{\D}\cap \Z^g$ such that
$x\in H_i+v$.
\end{enumerate}
The dimension of $V_{\D}$ is said to be the rank of $\D$ and is denoted by $\rk(\D)$. We say that $\D$ is non-degenerate (or simply that $\D$ is a {\bf lattice dicing}) if $\D$ has rank $g$, i.e. if $V_{\D}=\R^g$.

\end{defi}

\begin{remark}\label{R:latt-dicing}
\noindent
\begin{enumerate}[(i)]
\item The above definition of lattice dicing is equivalent to the definition in \cite[Sec. 2]{ER}.
\item If $\D$ is a generalized lattice dicing of rank $0\leq g'\leq g$ as above, then $V_{\D}\cap \Z^g\cong \Z^{g'}$ is a full dimensional lattice in $V_{\D}$ by (i), or equivalently $V_{\D}\cong (V_{\D}\cap \Z^g)\otimes_{\Z}\R\cong \R^{g'}$, and the hyperplanes $\{H_1\cap V_{\D}, \ldots, H_n\cap V_{\D}\}$ induce a lattice dicing of $V_{\D}\cong \R^{g'}$ (with respect to the lattice $V_{\D}\cap \Z^g\cong \Z^{g'}$), which we denote by $\D_{|V_{\D}}$.
\end{enumerate}
\end{remark}

To every simple unimodular matrix, it is possible to associate a generalized lattice dicing as follows.

\begin{lemmadefi}\label{D:latt-matrix}
Let $A\in M_{g, n}(\Z)$ be a simple unimodular matrix of rank $0\leq g'\leq g$. Denote its column vectors
by $\{v_1\,\ldots, v_n\}$ and, for each $1\leq i\leq n$, consider the central hyperplane $H_{v_i}$ of $\R^g$ defined
by $H_{v_i}:=\{x\in \R^g\: : \: v_i^t \cdot x=0 \}$. Then the collection
$\{H_{v_1},\ldots, H_{v_n}\}$ of central hyperplanes
determines a generalized lattice dicing $\D_A$ of $\R^g$ of rank $g'$.
\end{lemmadefi}
\begin{proof}
In the case where $A$ has maximum rank $g$, the result is proved in \cite[p. 462]{ER}.

In the general case, up to possibly replacing $A$ with a $\GL_g(\Z)$-equivalent matrix, we may assume that
$\displaystyle A=\left(\begin{matrix}
A'\\
0\\
\end{matrix}\right)$
where $A'\in M_{g',n}(\Z)$ is a simple unimodular matrix of maximal rank $g'$.
In this case, $V_{\D_A}=\langle e_1,\ldots,e_{g'}\rangle$ where $\{e_i\}$ is the standard basis of $\R^g$;
in particular, $V_{\D_A}$ is defined over $\Q$. Moreover, it is clear that the collection of hyperplanes
$\{H_{v_1}\cap V_{\D_A}, \ldots, H_{v_n}\cap V_{\D_A}\}$ defines the lattice dicing $\D_{A'}$. We deduce that
the collection $\{H_{v_1},\ldots, H_{v_n}\}$ satisfies properties (ii) and (iii) of Definition
\ref{D:latt-dicing} and we are done.
\end{proof}

We can now summarize the results of \cite{ER} in the following

\begin{fact}[Erdahl-Ryshkov]\label{F:Mat-Voro}
\noindent
\begin{enumerate}[(i)]
\item \label{F:Mat-Voro0} Every generalized lattice dicing of  $\R^g$ is of the form $\D_A$ for some
simple unimodular matrix $A\in M_{g,n}(\R)$.
\item \label{F:Mat-Voro1} For every simple unimodular matrix $A\in M_{g, n}(\Z)$, the generalized lattice dicing
$\D_A$ is a Delone subdivision and moreover we have that
$$\sigma(A)=\sigma_{\D_A}.$$
In particular, every cone of $\Mat$ is a cone of $\Voro$.
\item  \label{F:Mat-Voro2} For a cone $\sigma_D\in \Voro$, the following conditions are equivalent:
\begin{enumerate}[(a)]
\item $\sigma_D\in \Mat$;
\item $D$ is a generalized lattice dicing;
\item The extremal rays of $\sigma_D$ are generated by rank one positive semi-definite quadratic forms.
\end{enumerate}
\end{enumerate}
\end{fact}
\begin{proof}
Part \eqref{F:Mat-Voro0} is proved in \cite[p. 462]{ER} for lattice dicings (i.e. in the case of maximal rank $g$) and it is easily extended to generalized lattice dicings by looking at the lattice dicing $\D_{|V_{\D}}$ induced by $\D$ on $V_{\D}$ (see Remark \ref{R:latt-dicing}).

Under the assumption that $A$ has full rank $g$, part \eqref{F:Mat-Voro1} follows from \cite[Thm. 3.2 and Thm. 4.1]{ER} since
our cone $\sigma(A)$ (see Definition \ref{D:cone-mat}) coincides with the closure of the domain $\Phi(\D_A)$
of the lattice dicing $\D_A$ defined in \cite[Def. 3.1]{ER}. The extension to the general case follows easily as in Lemma-Definition \ref{D:latt-matrix}: up to replacing $A$ with a $\GL_g(\Z)$-equivalent matrix, we can write
 $\displaystyle A=\left(\begin{matrix}
A'\\
0\\
\end{matrix}\right)$
where $A'\in M_{g',n}(\Z)$ is a simple unimodular matrix of maximal rank $g'$ and then we deduce the assertion for $A$ from the analogous assertion for $A'$.

Part \eqref{F:Mat-Voro2}: the equivalence of (a) and (b) follows from part \eqref{F:Mat-Voro0}.


The equivalence of (a) and (c) is the content of \cite[Thm. 4.3]{ER}.
\end{proof}

There is another well-known characterization of the subcone $\Omat\subseteq \Ort$ in terms of Dirichlet-Voronoi polytopes.

\begin{remark}\label{R:zonotope}
A quadratic form $Q\in \Ort$ belongs to the matroidal subcone $\Omat$ if and only if its Dirichlet-Voronoi polytope
$\Vor(Q)$ is a \emph{zonotope}, i.e. a Minkowski sum of segments, or equivalently,
an affine projection of an hypercube. See e.g. \cite[Sec. 4.4]{BMV} and the references therein.
\end{remark}

\begin{example}\label{small-dim}
It is well-known that $\Mat$ is not pure-dimensional, i.e. the maximal cones of $\Mat$ are not of the same dimension
(see e.g. \cite[Chap. 4]{Val} and the references therein).
It is a classical result of Korkine-Zolotarev (\cite{KZ} or \cite[Thm. 5.2]{ER}) that, up to $\GL_g(\Z)$-equivalence, there is only one cone of $\Mat$ of maximum dimension $\binom{g+1}{2}$, namely the so-called {\bf principal cone} (or first perfect domain), which can be defined as (see \cite[Chap. 8.10]{NamT} and \cite[Chap. 2.3]{Val}):
\begin{equation}\label{E:prin}
\prin:=\{Q=(q_{ij})\in \Ort \: : \: q_{ij}\leq 0 \text{ for } i\neq j, \:
\sum_{j} q_{ij}\geq 0 \text{ for all } i\}.
\end{equation}
Indeed, the principal cone admits two well-known alternative descriptions:

\begin{enumerate}[(i)]
\item The $\GL_g(\Z)$-equivalence class $[\prin]$ of the principal
cone $\prin$ is equal to $\sigma(M(K_{g+1}))$, where $K_{g+1}$ is the complete
simple graph on $(g+1)$-vertices (see e.g. \cite[Lemma 6.1.3]{BMV} for a proof).
\item The interior of $\prin$ consists of all the quadratic forms $Q\in \O$ whose Dirichlet-Voronoi polytope
is normally equivalent to the permutahedron of dimension $g$ (see \cite[Ex. 0.10]{Zie}).
See e.g. \cite[Sec. 3.3.2]{Val} and the references therein.
\end{enumerate}
If $g=2,3$ then the principal cone $\prin$ is the unique maximal cone in $\Mat=\Voro$, up
to $\GL_g(\Z)$-equivalence (see \cite[Thm. 5.3]{ER}).

However, for $g\geq 4$, the matroidal decomposition $\Mat$ becomes quickly much smaller than $\Voro$ as $g$ grows
(and therefore the matroidal subcone $\Omat$ becomes smaller than $\Ort$). For small values of $g$, the number of equivalence classes of maximal cells of $\Mat$ and $\Voro$ are as follows (see \cite[Sec. 9]{DG} and \cite[Chap. 4]{Val}):
\begin{enumerate}[(i)]
\item  For $g=4$, $\Voro$ has $3$ maximal cells while $\Mat$ has two maximal cells of dimensions $10$ and $9$;
\item For $g=5$, $\Voro$ has $222$ maximal cells while $\Mat$ has $4$ maximal cells of dimensions
$15$, $12$, $12$ and $10$;
\item For $g=6$, $\Voro$ has more than $250,000$ maximal cells (although the exact number is still not known!)
    while $\Mat$ only $11$ maximal cells, $8$ of which have dimension $15$ and the others have dimensions $21$, $16$ and $12$.
\end{enumerate}
\end{example}

\subsection{$\Mat$ is contained in $\Per$.}

The aim of this subsection is to prove the following

\begin{thm}\label{T:Mat-cont-Perf}
We have that $\Mat\subseteq \Per$, i.e. every cone of $\Mat$ is a cone of $\Per$.
\end{thm}
\begin{proof}
We have to show that for any simple regular matroid $M$ of rank at most $g$, the equivalence class $\sigma(M)$
belongs to $\Per/\GL_g(\Z)$. The strategy is to prove this for graphic matroids, for cographic matroids and for the matroid $R_{10}$ and then apply Seymour's decomposition theorem (see Fact \ref{F:Seymour}). Let us first check the statement for $M$ belonging to each of the above classes.

\un{Graphic matroids}: Let $M=M(\Gamma)$ (see Definition \ref{grap&cogr-mat}), for $\Gamma$ a simple connected graph of  cogenus $g^*(\Gamma)=|V(\Gamma)|-1\leq g$. Clearly, $\Gamma$ can be obtained from the complete simple graph $K_{g+1}$ on $g+1$ vertices by deleting some of its edges. This means that, if we denote by $A(K_{g+1})\in M_{g, \binom{g+1}{2}}(\Z)$ a simple unimodular matrix representing the matroid $M(K_{g+1})$, then we can chose a simple unimodular matrix representing $\Gamma$ and having the form $A(\Gamma)=A(K_{g+1})\setminus I$, for a certain $I\subset \{1,\ldots, \binom{g+1}{2}\}$ which corresponds to the edges that we have deleted from $K_{g+1}$ in order to obtain $\Gamma$.
By Lemma \ref{L:prop-mat-cones}, $\sigma(A(\Gamma))$ is a face of $\sigma(A(K_{g+1}))$. Therefore, in order to prove that $\sigma(M(\Gamma))\in \Per/\GL_g(\Z)$, it is enough to prove that $\sigma(K_{g+1})\in \Per/\GL_g(\Z)$.
As observed in Example \ref{small-dim}, $\sigma(K_{g+1})$ is the equivalence class of the principal cone $\prin$ (see \eqref{E:prin}), which is well known to belong to $\Per$: indeed, it can be proven (see \cite[Sec. 8.10]{NamT} or \cite[Sec. 4.2]{Mar}) that
$$\prin=\sigma[Q_0] \: \text{ for } \: Q_0=
\left(\begin{matrix}
1 & 1/2 & \dotsm & 1/2\\
1/2 & \ddots  & \ddots & \vdots  \\
\vdots & \ddots & \ddots& 1/2 \\
1/2 & \dotsm & 1/2 & 1
\end{matrix}\right).$$

\un{Cographic matroids}: The fact that $\sigma(M^*(\Gamma))\in \Per/\GL_g(\Z)$ for any $3$-edge connected graph
$\Gamma$ of genus $g(\Gamma)\leq g$ was proved by Alexeev-Brunyate (see \cite[Thm. 5.6]{AB}).

\un{$R_{10}$}: Consider the simple totally unimodular matrix $A_{10}$ of rank $5$ from Definition \ref{D:R_10}
and its associated cone $\sigma(A_{10})\in \Mat\subseteq \Voro$. We have to prove that $\sigma(A_{10})\in \Per$.
Indeed, we will prove that $\sigma(A_{10})$ is a face of a top dimensional cone of $\Per$.

To this aim, consider the lattice $\bbD_5$ which, following the notations of \cite[Sec. 4.3]{Mar}, is defined to be the subgroup of $\Z^5$
consisting of all vectors $v=(v_1,\ldots, v_5)\in \Z^5$ such that $\sum_i v_i$ is even together with the restriction of the standard Euclidan quadratic form on $\R^5$.
If we denote by $\{\epsilon_1,\ldots,\epsilon_5\}$ the standard basis of $\Z^5$, then a basis for $\bbD_5$ is given by the vectors
$$e_i:=\epsilon_i+\epsilon_{i+1} \: \: \text{Êfor } i=1,\ldots, 5$$
where we have used the cyclic notation $\epsilon_{i+5}:=\epsilon_i$ for any $i\in \Z$.
With respect to the above basis, the positive definite quadratic form defining $\bbD_5$ is given by the matrix
$$Q_5=\left(\begin{matrix}
2 & 1 & 0 & 0 &  1\\
1 & 2 & 1 & 0 & 0   \\
0 & 1 & 2 & 1 & 0 \\
0 & 0 & 1 & 2 & 1 \\
1 & 0 & 0 & 1 & 2
\end{matrix}\right).$$
The quadratic form $Q_5$ is perfect (see \cite[Cor. 6.4.3]{Mar}) and the set $M(Q_5)$ of minimal integral non-zero vectors for $Q_5$ is given by the
$20$ vectors (see \cite[Sec. 4.3]{Mar})
\begin{equation}\label{E:min-vectors}
\begin{sis}
& e_i, \\
& f_i:=e_i-e_{i+1}, \\
& g_i:=e_i-e_{i+1}+e_{i+2}, \\
& h_i:=e_i-e_{i+1}+e_{i+2}-e_{i+3},
\end{sis}
\end{equation}
where $i=1,\ldots,5$ and we have used the cyclic notation $e_{i+5}:=e_i$ for any $i\in \Z$ (and similarly for $f_i$, $g_i$ and $h_i$).
Therefore, the cone $\sigma[Q_5]\in \Per$ has maximal dimension $15$ and it has $20$ extremal rays given by the rank one quadratic forms $\{e_i\cdot e_i^t, f_i\cdot f_i^t,
g_i\cdot g_i^t, h_i\cdot h_i^t \}_{i=1,\ldots, 5}$
associated to the above
elements of $M(Q_5)$. We claim that
\begin{equation*}\label{E:face}\tag{*}
\sigma(A_{10})\text{ is a face of } \sigma[Q_5],
\end{equation*}
which clearly would imply that $\sigma(A_{10})\in \Per$, as required.

Note that the columns of the matrix $A_{10}$ are exactly the $10$ vectors $\{e_i, g_i\}_{ i=1,\ldots,5}$; hence the extremal rays of $\sigma(A_{10})$ are generated by the $10$
rank one quadratic forms $\{e_i\cdot e_i^t,g_i\cdot g_i^t\}_{i=1,\cdots, 5}$. Therefore, in order to prove (*), we have to find a linear functional $H$ on the vector space $\R^{15}$
of quadratic forms on $\R^5$ that is a supporting hyperplane for $\sigma(A_{10})$, or in other words which satisfies (for any $i=1,\ldots,5$)
\begin{equation*}\label{E:supp-hyper}\tag{**}
\begin{sis}
& H(e_i\cdot e_i^t)=H(g_i\cdot g_i^t)=0, \\
& H(f_i\cdot f_i^t)<0, \\
& H(h_i\cdot h_i^t)<0.
\end{sis}
\end{equation*}
Consider the linear functional $H$ on $\R^{15}$ defined by
$$H\left(\sum_{1\leq i,j\leq 5}\alpha_{i,j}e_i\cdot e_j^t\right)=\sum_{i=1}^5 \alpha_{i,i+1}+2\sum_{i=1}^5 \alpha_{i,i+2}, $$
where $\alpha_{i,j}=\alpha_{j,i}\in \R$ with the usual cyclic convention 
$\alpha_{i+5,j}=\alpha_{i,j+5}:=\alpha_{i,j}$.
From the definition \eqref{E:min-vectors}, it follows easily that
$$\begin{sis}
& H(e_i\cdot e_i^t)=0, \\
& H(f_i\cdot f_i^t)=-2,\\
& H(g_i\cdot g_i^t)=0,\\
& H(h_i\cdot h_i^t)=-2.\\
\end{sis}$$
This implies that $H$ satisfies (**) and we are done.

\vspace{0,3cm}

In order to conclude the proof, it is enough, in view of Seymour's decomposition theorem (see Fact \ref{F:Seymour}),
to prove that if $M_1$ and $M_2$ are two simple regular matroids such that $\sigma(M_1), \sigma(M_2)\in \Per/\GL_g(\Z)$, then $\sigma(M_1\oplus_k M_2)\in \Per/\GL_g(\Z)$ for $k=1,2,3$  (see Definition \ref{D:sum-matroids}).
 From the definition of $\Per$ (see Subsection \ref{S:Perfect}), it follows that $\sigma(M)\in \Per/\GL_g(\Z)$ if and only if there exists  a simple totally unimodular matrix $A\in M_{g,n}(\Z)$
with column vectors $\{v_1,\ldots, v_n\}$ and a positive definite quadratic form $Q$ such that $M=M[A]$
and for any $\xi\in \Z^g\setminus 0$ it holds that $Q(\xi)\geq 1$ with equality if and only if $\pm \xi= v_i$ for some
$1\leq i\leq n$, or in the terminology of Definition \ref{D:well-suited} below, that $Q$ is well-suited for $A$. Therefore, we conclude using the Lemmas \ref{L:lem1}, \ref{L:lem2} and \ref{L:lem3} below.
\end{proof}

In order to simplify the statements of the Lemmas below, we introduce the following

\begin{defi}\label{D:well-suited}
Let $A\in M_{g,n}(\Z)$ be a simple totally unimodular matrix. We say that a symmetric matrix $Q\in M_{g,g}(\R)$ is   {\bf well-suited} for $A$ if $Q$ is positive definite and for any $\xi\in \Z^g\setminus 0$ it holds that $Q(\xi)\geq 1$ with equality if and only if $\pm \xi$ is equal to one of the column vectors of $A$.
\end{defi}

\begin{lemma}\label{L:lem1}
For $i=1,2$, let $A_i\in M_{g_i,n_i}(\Z)$ be a simple totally unimodular matrix and let $Q_i$ be a positive definite quadratic form of rank $g_i$ which is well-suited for $A_i$.
Then
$$Q=\left(\begin{matrix}
Q_1 & 0 \\
0 & Q_2\\
\end{matrix}\right) \text{ is well-suited for }
A=\left(\begin{matrix}
A_1 & 0 \\
0 & A_2\\
\end{matrix}\right). $$
\end{lemma}
\begin{proof}
It is clear that $Q$ is positive definite.
Take now an element $\eta\in\Z^{g_1+g_2}\setminus 0$ and write it as $\eta=(\xi_1,\xi_2)$ with $\xi_i\in \Z^{g_i}$.
Clearly $Q(\eta)=Q_1(\xi_1)+Q_2(\xi_2)$. Since at least one among $\xi_1$ and $\xi_2$ is non-zero because $\eta\neq 0$, we have that $Q(\eta)\geq 1$ with equality if and only if $\xi_1=0$ and $\pm \xi_2$ is a column vector of the matrix $A_2$ or viceversa, which is equivalent to say that $\pm \eta$ is a column vector of $A$.
\end{proof}


\begin{lemma}\label{L:lem2}
Consider two simple totally unimodular matrices of the form
$$A_1=\left(\begin{matrix}
B & 0 \\
b^t & 1\\
\end{matrix}\right),
\hspace{0,5cm}
A_2=\left(\begin{matrix}
c^t & 1 \\
C & 0\\
\end{matrix}\right),$$
where $B\in M_{g_1, n_1}(\Z)$, $C\in M_{g_2, n_2}(\Z)$ and $b, c$ are vectors.
Assume that $\ov{Q_i}$ is well-suited for $A_i$ for $i=1,2$. We can write
$$\ov{Q_1}=\left(\begin{matrix}
Q_1 & r_1 \\
r_1^t & 1\\
\end{matrix}\right) \text{ and }\:
\ov{Q_2}=\left(\begin{matrix}
1 & r_2^t \\
r_2 & Q_2\\
\end{matrix}\right)
$$
where $Q_i\in M_{g_i,g_i}(\R)$ and $r_i$ is a vector of length $g_i$ (for $i=1,2$). Then
$$\ov{Q}:=\left(\begin{matrix}
Q_1 & r_1 & r_1\cdot r_2^t  \\
r_1^t & 1 & r_2^t\\
r_2\cdot r_1^t  & r_2 & Q_2
\end{matrix}\right) \text{ is well-suited for }
A=\left(\begin{matrix}
B & 0 & 0 \\
b^t & c^t & 1\\
0 & C & 0
\end{matrix}\right).
$$
\end{lemma}
\begin{proof}
The fact that $\ov{Q_i}$ (for $i=1,2$) can be written in the required form follows from the fact that $\ov{Q_i}$
takes the value $1$ on the last column of $A_i$ since $\ov{Q_i}$ is well-suited for $A_i$.

Consider now a vector $(\xi_1,x,\xi_2)$ where
$\xi_i=(\xi_i^1,\dots,\xi_i^{g_i})\in \R^{g_i}$ (for $i=1,2$) and $x\in \R$.
Then, using block matrix multiplication, we have that
\begin{equation*}
\begin{sis}
 (\xi_1,x)^t\ov{Q_1}(\xi_1,x)&= \xi_1^tQ_1\xi_1+2\xi_1^tr_1x  +x^2,\\
(x,\xi_2)^t\ov{Q_2}(x,\xi_2)&= x^2 +2xr_2^t.\xi_2+\xi_2^tQ_2\xi_2,\\
 (\xi_1,x,\xi_2)^t\ov{Q}(\xi_1,x,\xi_2)& =\xi_1^t Q_1\xi_1+2\xi_1^tr_1x+2\xi_1^tr_1r_2^t\xi_2
 +x^2+2xr_2^t\xi_2+\xi_2^tQ_2\xi_2,
\end{sis}
\end{equation*}
which we rewrite as
\begin{equation}\label{E:values-forms}
\begin{sis}
& \ov{Q_1}(\xi_1,x)=Q_1(\xi_1)+2x\langle r_1, \xi_1\rangle +x^2, \\
& \ov{Q_2}(x,\xi_2)=Q_2(\xi_2)+2x\langle r_2, \xi_2\rangle +x^2, \\
& \ov{Q}(\xi_1,x,\xi_2)=Q_1(\xi_1)+Q_2(\xi_2)+ 2\langle r_1, \xi_1\rangle \langle r_2, \xi_2\rangle+
2x\left[ \langle r_1, \xi_1\rangle +\langle r_2, \xi_2\rangle\right]+ x^2,
\end{sis}
\end{equation}
where $\langle \: ,\: \rangle$ denotes the usual scalar product of vectors.

For a fixed value $\xi_i\in \R^{g_i}$, the minimum of $\ov{Q_i}$ considered as a function on $x$ is attained at
$-\langle \xi_i,r_i\rangle$ and it is equal to $Q_i(\xi_i)-\langle \xi_i,r_i\rangle^2$. Indeed, for any quadratic real function $f$ of the form $f(x)=x^2+2bx+c$, for real numbers $a$ and $b$, the minimum of $f$ is attained when $x=-b$ and it is equal to $c-b^2$. Therefore, since $\ov{Q_i}$ is assumed to be positive definite, we get that (for $i=1,2$)
\begin{equation}\label{E:min-Q_i}
Q_i(\xi_i)-\langle \xi_i, r_i\rangle^2\geq 0 \text{ with equality if and only if } \xi_i=0.
\end{equation}
Similarly, for fixed values $(\xi_1,\xi_2)\in \R^{g_1+g_2}$, the minimum of $\ov{Q}$ considered as a function on $x$
is attained at $x_0=-\langle \xi_1,r_1\rangle-\langle \xi_2,r_2\rangle$ and it is equal to
\begin{equation}\label{E:min-Q}
\min_{x\in \R} \ov{Q}(\xi_1,x,\xi_2)= \ov{Q}(\xi_1,x_0,\xi_2)=Q_1(\xi_1)-\langle \xi_1, r_1\rangle^2+
Q_2(\xi_2)-\langle \xi_2, r_2\rangle^2.
\end{equation}
Using \eqref{E:min-Q_i}, we get that $\min_{x\in \R} \ov{Q}(\xi_1,x,\xi_2)\geq 0$ with equality if and only if
$(\xi_1,\xi_2)=(0,0)$, which proves that $\ov{Q}$ is positive definite. It remains to show that $\ov Q$ is well-suited for $A$, i.e., that for any $(\xi_1,x,\xi_2)\in\Z^{g_1+g_2+1}$, $\ov Q(\xi_1,x,\xi_2)\geq 1$ with equality if and only if $\pm (\xi_1,x,\xi_2)$ is equal to a column vector of $A$.

Fix $(\xi_1,x,\xi_2)\in \Z^{g_1+g_2+1}$. Start by noticing that, since no column vectors of $B$ and of $C$ can be equal to $0$ (see Remark \ref{R:sum}), the vector  $\pm (\xi_1,x,\xi_2)$ is a column vector of $A$ if and only if $\xi_1=0$ and $\pm (x,\xi_2)$ is a column vector of $A_2$ or if $\xi_2=0$ and $\pm (\xi_1,x)$ is a column vector of $A_1$. If $\xi_1=0$ then $\ov{Q}(0, x,\xi_2)=\ov{Q_2}(x,\xi_2)$  by \eqref{E:values-forms}. Now, since $\ov{Q_2}$ is well-suited for $A_2$, we get that
$\ov{Q}(0,x,\xi_2)=\ov{Q_2}(x,\xi_2)\geq 1$ for any $(x,\xi_2)\in \Z^{g_2+1}\setminus 0$ with equality if and only if $\pm (x,\xi_2)$ is a column vector of $A_2$, or equivalently, if and only if $\pm (0,x,\xi_2)$ is a column vector of $A$. We get the same conclusions if $\xi_2=0$.

Therefore, it remains to show that if $\xi_i\in \Z^{g_i}\setminus 0$ for $i=1,2$
 and $x\in \Z$ then
$\ov{Q}(\xi_1,x,\xi_2)>1$. Using \eqref{E:min-Q}, this is a consequence of the following

\un{CLAIM}: If $\xi_i\in \Z^{g_i}\setminus 0$ then $Q_i(\xi_i)-\langle \xi_i, r_i\rangle^2\geq 3/4$ for $i=1,2$.

Let us prove the Claim for $i=1$ (the case $i=2$ being analogous). As observed before, we have that
\begin{equation}\label{E:equa1}
Q_1(\xi_1)-\langle \xi_1, r_1\rangle^2=\min_{x\in \R} \ov{Q_1}(\xi_1,x)=\ov{Q_1}(\xi_1,-\langle \xi_1,r_1\rangle).
\end{equation}
Let $x_{\rm min}=-\langle \xi_1,r_1\rangle$ and denote by $M=[x_{\rm min}]\in \Z$ its integer part.  Then we have that
\begin{equation}\label{E:equa2}
\ov{Q_1}(\xi_1,M), \ov{Q_1}(\xi_1,M+1)\geq 1,
\end{equation}
by our original assumptions on $\ov{Q_1}$ and the fact that $\xi_1\in \Z^{g_1}\setminus 0$. Using
\eqref{E:values-forms} we compute
\begin{equation}\label{E:equa3}
\begin{sis}
& \ov{Q_1}(\xi_1,M)-\ov{Q_1}(\xi_1,x_{\rm min})=2(M-x_{\rm min})(-x_{\rm min})+M^2-x_{\rm min}^2=(M-x_{\rm min})^2, \\
& \ov{Q_1}(\xi_1,M+1)-\ov{Q_1}(\xi_1,x_{\rm min})=2(M+1-x_{\rm min})(-x_{\rm min})+(M+1)^2-x_{\rm min}^2=\\&\hspace{5cm}=(M+1-x_{\rm min})^2. \\
\end{sis}
\end{equation}
Equation \eqref{E:equa2} together with \eqref{E:equa3} gives that
\begin{equation}\label{E:equa4}
\begin{sis}
& \ov{Q_1}(\xi_1,x_{\rm min})\geq 1-(M-x_{\rm min})^2, \\
& \ov{Q_1}(\xi_1,x_{\rm min})\geq 1-(M+1-x_{\rm min})^2. \\
\end{sis}
\end{equation}
Putting together \eqref{E:equa1} and \eqref{E:equa4}, we deduce that
$$Q_1(\xi_1)-\langle \xi_1, r_1\rangle^2= \ov{Q_1}(\xi_1,x_{\rm min})\geq \max\{1-(M-x_{\rm min})^2,
1-(M+1-x_{\rm min})^2 \}=$$
$$=1- \min\{M-x_{\rm min},M+1-x_{\rm min} \}^2\geq 1-\left(\frac{1}{2}\right)^2=\frac{3}{4}.$$

\end{proof}

\begin{lemma}\label{L:lem3}
Consider two simple totally unimodular matrices of the form
$$A_1=A_1=\left(\begin{matrix}
B & 0 & 0 & 0 \\
b_1^t & 1 & 0& 1\\
b_2^t & 0 & 1 & 1
\end{matrix}\right),
\hspace{0,5cm}
A_2=\left(\begin{matrix}
c_1^t & 1 & 0 & 1\\
c_2^t& 0 & 1 & 1\\
C & 0 & 0 & 0 \\
\end{matrix}\right),$$
where $B\in M_{g_1, n_1}(\Z)$, $C\in M_{g_2, n_2}(\Z)$ and $b_1, b_2, c_1, c_2$ are vectors.
Assume that $\ov{Q_i}$ is well-suited for $A_i$ for $i=1,2$.
We can write
$$\ov{Q_1}=\left(\begin{matrix}
Q_1 & r_1 & s_1\\
r_1^t & 1 & -1/2\\
s_1^t & -1/2 & 1
\end{matrix}\right) \text{ and } \:
\ov{Q_2}=\left(\begin{matrix}
 1 & -1/2 & r_2^t\\
-1/2 & 1 & s_2^t \\
r_2 & s_2 & Q_2\\
\end{matrix}\right)
$$
where $Q_i\in M_{g_i,g_i}(\R)$ and $r_i, s_i$ are vectors of length $g_i$ (for $i=1,2$). Then
$$\ov{Q}:=\left(\begin{matrix}
Q_1 & r_1 & s_1 & M  \\
r_1^t & 1 & -1/2 & r_2^t\\
s_1^t & -1/2 & 1 & s_2^t\\
M^t & r_2 & s_2 & Q_2
\end{matrix}\right) \text{ is well-suited for }
A=\left(\begin{matrix}
B & 0 & 0 & 0 & 0\\
b_1^t & c_1^t & 1 & 0 & 1\\
b_2^t & c_2^t & 0 & 1 & 1\\
0 & C & 0 & 0 & 0
\end{matrix}\right),
$$
where $\displaystyle M:=\frac{4r_1r_2^t+ 4s_1 s_2^t+ 2r_1 s_2^t+ 2s_1 r_2^t}{3}\in M_{g_1,g_2}(\R)$.
\end{lemma}
\begin{proof}
The fact that $\ov{Q_i}$ (for $i=1,2$) can be written in the required form follows from the fact that $\ov{Q_i}$
takes value $1$ on the last three columns of $A_i$ since $\ov{Q_i}$ is well-suited for $A_i$.


Consider a vector $(\xi_1,x,y,\xi_2)$ where
$\xi_i\in \R^{g_i}$ (for $i=1,2$) and $x, y\in \R$.
Then, using block matrix multiplication, we have that
\begin{equation*}
\begin{sis}
 (\xi_1,x,y)^t\ov{Q_1}(\xi_1,x, y)&= \xi_1^tQ_1\xi_1+2\xi_1^tr_1x +2\xi_1^ts_1y +x^2 +2x\left(-\frac12\right)y+y^2, \\
(x,y,\xi_2)^t\ov{Q_2}(x,y,\xi_2)&= x^2 +2x\left(-\frac12\right)y+2xr_2^t\xi_2+y^2+2ys_2^t\xi_2+\xi_2^tQ_2\xi_2,\\
 (\xi_1,x,y,\xi_2)^t\ov{Q}(\xi_1,x,y,\xi_2)& =\xi_1^t Q_1\xi_1+2\xi_1^tr_1x+2\xi_1^ts_1y+2\xi_1^tM\xi_2\\
 &+x^2+2x\left(-\frac 12\right)y+2xr_2^t\xi_2+y^2+2ys_2^t\xi_2+\xi_2^tQ_2\xi_2,
\end{sis}
\end{equation*}
which we rewrite as
\begin{equation}\label{E:values-forms2}
\begin{sis}
 \ov{Q_1}(\xi_1,x, y)&= Q_1(\xi_1)+2x\langle r_1, \xi_1\rangle +2y\langle s_1,\xi_1\rangle +x^2 -xy+y^2, \\
 \ov{Q_2}(x, y,\xi_2)& = Q_2(\xi_2)+2x\langle r_2, \xi_2\rangle+2y\langle s_2,\xi_2\rangle +x^2 -xy+y^2 , \\
 \ov{Q}(\xi_1,x,y,\xi_2)& = Q_1(\xi_1)+Q_2(\xi_2)+2 \xi_1^t M\xi_2+2x\left[ \langle r_1, \xi_1\rangle +\langle r_2, \xi_2\rangle\right]\\
 &+
2y\left[ \langle s_1, \xi_1\rangle +\langle s_2, \xi_2\rangle\right]+x^2-xy+y^2,
\end{sis}
\end{equation}
where $\langle , \rangle$ denotes the usual scalar product of vectors.
Let $f:\R^2\to \R$ be a quadratic function of the form $f(x,y)=x^2-xy+y^2+2ax+aby+c$, where $a,b$ and $c$ are real numbers.
Then an easy calculation shows that the minimum value of $f$ is attained  when
\begin{equation}\label{E:gen-min-values}
\begin{sis}
x=-\frac 43 a-\frac 23b \\
y=-\frac 23a-\frac 43b
\end{sis}
\end{equation}
and it is equal to $-\frac 43(a^2+b^2+ab)+c$.
So, by \eqref{E:gen-min-values}, for a fixed value $\xi_i\in \R^{g_i}$, the minimum of $\ov{Q_i}$ considered as a function on $x$ and $y$ is attained at
\begin{equation}\label{E:min-values}
\begin{sis}
& x^i_{\rm min}=-\frac{4}{3}\langle r_i,\xi_i\rangle -\frac{2}{3}\langle s_i,\xi_i\rangle,\\
& y^i_{\rm min}=-\frac{2}{3}\langle r_i,\xi_i\rangle -\frac{4}{3}\langle s_i,\xi_i\rangle,\\
\end{sis}
\end{equation}
and it is equal to $\displaystyle Q_i(\xi_i)-\frac{4}{3}\left[\langle \xi_i,r_i\rangle^2+ \langle \xi_i,s_i\rangle^2+\langle \xi_i,r_i\rangle \langle \xi_i,s_i\rangle\right]$. Therefore, since $\ov{Q_i}$ is assumed to be positive definite, we get that (for $i=1,2$)
\begin{equation}\label{E:min-Q_i2}
 Q_i(\xi_i)-\frac{4}{3}\left[\langle \xi_i,r_i\rangle^2+ \langle \xi_i,s_i\rangle^2+\langle \xi_i,r_i\rangle \langle \xi_i,s_i\rangle\right]\geq 0
 \text{ with equality if and only if } \xi_i=0.
\end{equation}
Similarly, for fixed values $(\xi_1,\xi_2)\in \R^{g_1+g_2}$, the minimum of $\ov{Q}$ considered as a function on $x$ and $y$  is attained at
$$\begin{sis}
& x_0=-\frac{4}{3}\left[\langle r_1,\xi_1\rangle+\langle r_2,\xi_2\rangle \right]-
\frac{2}{3}\left[\langle s_1,\xi_1\rangle + \langle s_2,\xi_2\rangle \right],\\
& y_0=-\frac{2}{3}\left[\langle r_1,\xi_1\rangle+\langle r_2,\xi_2\rangle \right]-
\frac{4}{3}\left[\langle s_1,\xi_1\rangle + \langle s_2,\xi_2\rangle \right],\\
\end{sis}$$
and it is equal to

\begin{equation}
\begin{aligned}
&\ov{Q}(\xi_1,x_0,y_0,\xi_2)=-\frac 43\left[(\langle r_1,\xi_1\rangle+\langle r_2,\xi_2\rangle)^2+(\langle s_1,\xi_1\rangle+\langle s_2,\xi_2\rangle)^2+\right.\\
&\left.(\langle r_1,\xi_1\rangle+\langle r_2,\xi_2\rangle)(\langle s_1,\xi_1\rangle+\langle s_2,\xi_2\rangle)\right]+2\xi_1^tM\xi_2+Q_i(\xi_1)+Q_2(\xi_2)\\
&=\sum_{i=1}^2\left\{ Q_i(\xi_i)-\frac{4}{3}\left[\langle \xi_i,r_i\rangle^2+B+2\xi_1^tM\xi_2\langle \xi_i,s_i\rangle^2+\langle \xi_i,r_i\rangle \langle \xi_i,s_i\rangle\right]\right\},
\end{aligned}
\end{equation}
where $B:=2\langle r_1,\xi_1\rangle\langle r_2,\xi_2\rangle +2\langle s_1,\xi_1\rangle\langle s_2,\xi_2\rangle +\langle r_1,\xi_1\rangle\langle s_2,\xi_2\rangle +\langle r_2,\xi_2\rangle\langle s_1,\xi_1\rangle$. We claim that $B+2\xi_1^tM\xi_2=0$. Given a vector $v$, denote by $v^j$ the $j$-th entry of $v$. Then our claim follows from the easy observation that the coefficient of $\xi_1^j\xi_2^k$ in the expression $B$ is equal to
$$-\frac 43\left(2r_1r_2^k+2s_1s_2^k+r_1^js_2^k+r_2^ks_1^j\right),$$
and thus is opposite to the coefficient of $\xi_1^j\xi_2^k$ in $2\xi_1^tM\xi_2$, which turns out to be
$$2\left(\frac 43r_1r_2^k+\frac 43s_1s_2^k+\frac 23r_1^js_2^k+\frac 23s_1^jr_2^k\right) .$$
In conclusion, we have that
\begin{equation}\label{E:min-Q2}
\begin{aligned}
& \min_{x,y\in \R} \ov{Q}(\xi_1,x,y,\xi_2)= \ov{Q}(\xi_1,x_0,y_0,\xi_2)= \sum_{i=1}^2\left\{ Q_i(\xi_i)-\frac{4}{3}\left[\langle \xi_i,r_i\rangle^2+ \right.\right.\\
&\left.\left.\langle \xi_i,s_i\rangle^2+\langle \xi_i,r_i\rangle \langle \xi_i,s_i\rangle\right]\right\}=
\min_{x,y\in \R} \ov{Q_1}(\xi_1,x,y)+ \min_{x,y\in \R} \ov{Q_2}(x,y,\xi_2).
\end{aligned}
\end{equation}
Using \eqref{E:min-Q_i2}, we get that $\displaystyle \min_{x,y\in \R} \ov{Q}(\xi_1,x,y,\xi_2)\geq 0$ with equality if and only if
$(\xi_1,\xi_2)=(0,0)$, which proves that $\ov{Q}$ is positive definite. It remains to show that $\ov Q$ is well-suited for $A$, i.e., that for any $(\xi_1,x,y,\xi_2)\in\Z^{g_1+g_2+2}$, $\ov Q(\xi_1,x,y,\xi_2)\geq 1$ with equality if and only if $\pm (\xi_1,x,y,\xi_2)$ is equal to a column vector of $A$.

Fix $(\xi_1,x,y,\xi_2)\in \Z^{g_1+g_2+2}$.
Using the same type of argumentation as in the proof of Lemma \ref{L:lem2}, we start by noticing that, since no column vectors of $B$ and of $C$ can be equal to $0$ (see Remark \ref{R:sum}), the vector  $\pm (\xi_1,x,y,\xi_2)$ is a column vector of $A$ if and only if $\xi_1=0$ and $\pm (x,y,\xi_2)$ is a column vector of $A_2$ or if $\xi_2=0$ and $\pm (\xi_1,x,y)$ is a column vector of $A_1$.
If $\xi_1=0$ then $\ov{Q}(0, x,y,\xi_2)=\ov{Q_2}(x,y,\xi_2)$  by \eqref{E:values-forms2}. Now, since $\ov{Q_2}$ is well-suited for $A_2$, we get that
$\ov{Q}(0,x,y,\xi_2)=\ov{Q_2}(x,y,\xi_2)\geq 1$ for any $(x,y,\xi_2)\in \Z^{g_2+2}\setminus 0$ with equality if and only if $\pm (x,y,\xi_2)$ is a column vector of $A_2$ or, equivalently, if and only if $\pm (0,x,y,\xi_2)$ is a column vector of $A$. We get the same conclusions if $\xi_2=0$.

Therefore, it remains to show that if $\xi_i\in \Z^{g_i}\setminus 0$ for $i=1,2$ and $x,y\in \Z$ then
$\ov{Q}(\xi_1,x,y,\xi_2)>1$. Using \eqref{E:min-Q2}, this is a consequence of the following

\un{CLAIM}: If $\xi_i\in \Z^{g_i}\setminus 0$ then $Q_i(\xi_i)-\frac{4}{3}\left[\langle \xi_i,r_i\rangle^2+ \langle \xi_i,s_i\rangle^2+\langle \xi_i,r_i\rangle \langle \xi_i,s_i\rangle\right]\geq \frac{3}{4}$, for $i=1,2$.

Let us prove the Claim for $i=1$ (the case $i=2$ being analogous). As observed before, we have that
\begin{equation}\label{E:equa1bis}
Q_1(\xi_1)-\frac{4}{3}\left[\langle \xi_1,r_1\rangle^2+ \langle \xi_1,s_1\rangle^2+\langle \xi_1,r_1\rangle \langle \xi_1,s_1\rangle\right]=\ov{Q_1}(\xi_1,x_{\rm min}^1, y_{\rm min}^1),
\end{equation}
where $x_{\rm min}^1$ and $y_{\rm min}^1$ are given in \eqref{E:min-values}. Let $M_1=[x^1_{\rm min}]\in \Z$ and
$M_2=[y^1_{\rm min}]\in \Z$ be their integer parts.  Then we have that
\begin{equation}\label{E:equa2bis}
\ov{Q_1}(\xi_1,M_1,M_2), \ov{Q_1}(\xi_1,M_1+1,M_2), \ov{Q_1}(\xi_1,M_1,M_2+1), \ov{Q_1}(\xi_1,M_1+1,M_2+1)\geq 1,
\end{equation}
by our original assumptions on $\ov{Q_1}$ and the fact that $\xi_1\in \Z^{g_i}\setminus 0$ and $M_1,M_2\in \Z$.
Now, from equation \eqref{E:values-forms2} we have that for any $x,y\in \R$
\begin{equation}\label{E:equa3bis}
\begin{aligned}
\ov{Q_1}(\xi_1,x,y)-\ov{Q_1}(\xi_1,x^1_{\rm min},y^1_{\rm min})=2(x-x^1_{\rm min})\langle r_1,\xi_1 \rangle+(y-y^1_{\rm min})\langle s_1,\xi_1\rangle\\
+x^2-xy+y^2-(x^1_{\rm min})^2+x^1_{\rm min}y^1_{\rm min}-(y^1_{\rm min})^2\\
=(x-x^1_{\rm min})^2-(x-x^1_{\rm min})
(y-y^1_{\rm min})+(y-y^1_{\rm min})^2,
\end{aligned}
\end{equation}
where the last equality follows from the fact that
\begin{equation*}
\begin{sis}
& \langle r_1,\xi_1\rangle=-2x^1_{\rm min}+y^1_{\rm min}\\
& \langle s_1,\xi_1\rangle=x^1_{\rm min}-2y^1_{\rm min}
\end{sis}
\end{equation*}
which we can easily deduce from  \eqref{E:min-values}.
Putting together \eqref{E:equa2bis} and \eqref{E:equa3bis}, we deduce that, if $x$ takes the value of either $M_1$ or $M_1+1$ and if $y$ takes the value of either $M_2$ or $M_2+1$, then
$$\ov{Q_1}(\xi_1,x^1_{\rm min},y^1_{\rm min})\geq 1-\left[(x-x^1_{\rm min})^2-(x-x^1_{\rm min})
(y-y^1_{\rm min})+(y-y^1_{\rm min})^2\right]$$
which implies that
$$\ov{Q_1}(\xi_1,x^1_{\rm min},y^1_{\rm min})\geq
1- \min_{\stackrel{x=M_1,M_1+1}{y=M_2,M_2+1}}\{(x-x^1_{\rm min})^2-(x-x^1_{\rm min})
(y-y^1_{\rm min})+(y-y^1_{\rm min})^2\}.
$$
The minimum appearing in the last equation will be at most equal to $1/4$, which will be the case if
 $x^1_{\rm min}=M_1+1/2$ and
$y^1_{\rm min}=M_2+1/2$.
Therefore we get that
\begin{equation*}
\ov{Q_1}(\xi_1,x^1_{\rm min},y^1_{\rm min})\geq 1-\frac{1}{4}=\frac{3}{4},
\end{equation*}
which, combined with \eqref{E:equa1bis}, concludes the proof of the Claim.
\end{proof}



\subsection{$\Mat$ is the intersection of $\Voro$ and $\Per$.}\label{S:inter-Vor-Per}

The aim of this subsection is to prove the following

\begin{prop}\label{P:inter-Mat-Perf}
We have the following
$$\Voro\cap \Per\subseteq \Mat,$$
i.e. if $\sigma$ is a cone of $\Voro$ and of $\Per$ then $\sigma$ is a cone of $\Mat$.
\end{prop}
\begin{proof}
Let $\sigma\in \Voro\cap \Per$.  The fact that $\sigma\in \Per$ implies, by Remark \ref{R:Per-simpli}, that the extremal rays of $\sigma$ are generated by positive semi-definite
quadratic forms of rank one. Fact \ref{F:Mat-Voro}\eqref{F:Mat-Voro2}, together with the hypothesis that $\sigma\in \Voro$, now implies that $\sigma\in \Mat$.
\end{proof}

By combining Fact \ref{F:Mat-Voro}\eqref{F:Mat-Voro1}, Theorem \ref{T:Mat-cont-Perf} and Proposition \ref{P:inter-Mat-Perf}, we deduce the following

\begin{cor}\label{C:main-result}
We have that
$$\Voro\cap \Per= \Mat,$$
i.e. a cone $\sigma$ belongs to $\Voro$ and $\Per$ if and only if it belongs to $\Mat$.
\end{cor}

Combining Corollary \ref{C:main-result} with Example \ref{small-dim}, we deduce the following classical result of
Dickson (\cite[Thm. 2]{Dic}):

\begin{cor}[Dickson]
The principal cone $\prin$ is the unique cone of (maximal) dimension $\binom{g+1}{2}$, up to $\GL_g(\Z)$-equivalence, which is contained in $\Voro$ and $\Per$.
\end{cor}

\section{Toroidal compactifications of $\Ag$}\label{S:toroidal}

\subsection{Preliminaries on $\AgPer$ and $\AgVor$}

From the general theory of toroidal compactifications (see \cite{AMRT} for the general case of bounded symmetric domains and \cite{NamT} for the special case of the Siegel upper half space), it follows that to each
admissible decomposition $\Sigma$ of $\Ort$ (in the sense of Definition \ref{decompo}) it is associated a toroidal compactification $\ov{\Ag}^{\Sigma}$ of the moduli space $\Ag$ of principally polarized abelian varieties of dimension $g$, i.e. a complete normal variety $\ov{\Ag}^{\Sigma}$ containing $\Ag$ as a dense open subset and such that the pair
$(\Ag,\ov{\Ag}^{\Sigma})$ is \'etale locally isomorphic to a torus inside a complete toric variety.
By construction, the toroidal compactification $\ov{\Ag}^{\Sigma}$ comes with a stratification into locally closed
subsets which are naturally in order-reversing bijection (with respect to the order relation given by the closure) with the $\GL_g(\Z)$-equivalence classes of the relative interiors of the cones in $\Sigma$. For example, the origin of $\Ort$ (which is the unique zero dimensional cone in every admissible decomposition $\Sigma$) corresponds to the open subset $\Ag$ (which is the unique stratum of $\ov{\Ag}^{\Sigma}$ of maximal dimension $\binom{g+1}{2}$), while the maximal dimensional cones in $\Sigma$ correspond to the zero dimensional strata of $\ov{\Ag}^{\Sigma}$.

We will be interested in the toroidal compactifications of $\Ag$ associated to the perfect cone decomposition and to
the 2nd Voronoi decomposition, which are called, respectively, the perfect toroidal compactification and the 2nd Voronoi compactification of $\Ag$ and
are denoted by $\AgPer$ and $\AgVor$, respectively. It is known that $\AgPer$ and $\AgVor$ are projective (for $\AgPer$ this follows easily from the construction, see \cite[Chap. 8]{NamT} for details; for $\AgVor$ this is a non-trivial result of Alexeev, see
\cite[Cor. 5.12.8]{alex1}). Note that since $\Per$ has non simplicial cones for $g\geq 4$ (see Remark \ref{R:Per-simpli}) and similarly $\Voro$ has non simplicial cones for $g\geq 5$
(see Remark \ref{R:Voro-nonsimpli}), the compactifications $\AgPer$ and $\AgVor$ do not have finite quotient singularities if, respectively, $g\geq 4$ or $g\geq 5$.

These two toroidal compactifications of $\Ag$ have a special importance due to the following

\begin{fact}\label{F:comp-Tor}
\noindent
\begin{enumerate}[(i)]
\item (Shepherd-Barron \cite{SB}) $\AgPer$ is the canonical model of $\Ag$ for $g\geq 12$.
\item (Alexeev \cite{alex1}) $\AgVor$ is the normalization of the main irreducible component of Alexeev's
    moduli space $\ov{AP_g}$ of stable semiabelic pairs, which provides a \emph{modular} compactification of $\Ag$.
\item (Mumford-Namikawa \cite{nam2}, Alexeev \cite{alex}, Alexeev-Brunyate \cite{AB}) The Torelli map
$$ \tg : \Mg\to \Ag,  $$
sending a curve $X\in \Mg$ into its polarized Jacobian $(\Jac(X), \Theta_X)\in \Ag$, extends to regular maps
\begin{equation}\label{comp-Torelli}
\tgb^V: \Mgb\to \AgVor \hspace{0,3cm} \text{ and } \hspace{0,3cm} \tgb^P: \Mgb\to \AgPer,
\end{equation}
where $\Mgb$ is the Deligne-Mumford (see \cite{DM}) compactification of $\Mg$ via stable curves.
\item (Alexeev-Brunyate \cite{AB}) \label{F:comp-Tor4}
$\AgVor$ and $\AgPer$ contain a common open subset $\Agcogra$ given by the union of the strata corresponding to the
$\GL_g(\Z)$-equivalence classes of (cographic) cones $\sigma(M^*(\Gamma))$, where $\Gamma$ varies among all $3$-egde connected graphs of genus at most $g$.

Moreover, $\Agcogra$ contains the images of the morphisms $\tgb^V$ and $\tgb^P$.

\end{enumerate}
\end{fact}
We mention that the compactified Torelli map $\tgb^V$ admits a very nice modular description due to Alexeev (see \cite[Sec. 5]{alex}), which has been used by Caporaso-Viviani \cite{CV2} to describe its geometric fibers. On the other hand, the map $\tgb^P$ has been used by Gibney \cite{Gib} to find some interesting semi-ample divisors on $\Mgb$.

\subsection{Comparing $\AgPer$ and $\AgVor$}

The aim of this subsection is to compare the perfect compactification $\AgPer$ with the 2nd Voronoi compactification
$\AgVor$.  Let us first introduce a special sublocus of $\AgVor$.

\begin{defi}\label{D:mat-locus}
Let $\Agmat$ be the open subset of $\AgVor$ given by the union of the strata of $\AgVor$ corresponding to the
$\GL_g(\Z)$-equivalence classes of cones belonging to $\Mat\subseteq \Voro$.
We call $\Agmat$ the {\bf matroidal locus} of $\AgVor$.
\end{defi}
The fact that $\Agmat$ is an open subset of $\AgVor$ follows from the fact that $\Mat\subseteq \Voro$ is closed under taking faces of cones. Note that $\Agmat$ has abelian finite quotient singularities since $\Mat$ is made of simplicial cones.

We can now state the main result of this subsection, which answers positively to a question of Alexeev-Brunyate
in \cite[6.3]{AB}. In particular, part \eqref{L:part3} of the Theorem below is an extension of Fact \ref{F:comp-Tor}\eqref{F:comp-Tor4} (due to Alexeev-Brunyate) since $\Agmat$ clearly contains $\Agcogra$.

\begin{thm}\label{T:compa-toro}
\noindent
\begin{enumerate}[(i)]
\item \label{L:part1} $\Agmat$ is the biggest open subset of $\AgVor$ where the rational map $\AgVor\stackrel{\tau}{\dashrightarrow} \AgPer$ is defined and is an isomorphism.
\item \label{L:part2} $\tau(\Agmat)$ is the biggest open subset of $\AgPer$ where the rational map $\AgPer \stackrel{\tau^{-1}}{\dashrightarrow} \AgVor$ is defined.
\item \label{L:part3} The compactified Torelli maps $\tgb^P$ and $\tgb^V$ fit into the following commutative diagram
    $$\xymatrix{
    & \Agmat \ar@{^{(}->}[r] \ar[dd]^{\tau}_{\cong}& \AgVor \ar@{-->}[dd]^{\tau} \\
    \Mgb \ar[ru]^{\tgb^V} \ar[rd]_{\tgb^P} & & \\
    & \tau(\Agmat) \ar@{^{(}->}[r] & \AgPer\\
    }$$
\end{enumerate}
\end{thm}
\begin{proof}
The proof follows from the combinatorial results of the previous sections together with standard facts from the theory
of toroidal compactifications.

Part \eqref{L:part1} follows from Corollary \ref{C:main-result}.

Part \eqref{L:part2} follows from Lemma \ref{L:Per-Vor} below.

Part \eqref{L:part3}. The map $\tgb^V$ sends the stratum of $\Mgb$ corresponding to stable curves with dual graph $\Gamma$ into the stratum of $\AgVor$ corresponding to
$\sigma(M^*(\Gamma))\in \Mat/\GL_g(\Z)\subseteq \Voro/\GL_g(\Z)$ (see \cite[Thm. 3.11 and Thm. 4.1]{alex}). The same is true for the map $\tgb^P$ by \cite[Thm. 3.7 and Thm. 5.6]{AB}.
By general results on  toroidal compactifications (see e.g. \cite[Thm. 3.2]{AB}), it follows now that $\Im(\tgb^V)\subseteq \Agmat\subseteq \AgVor$,
$\Im(\tgb^P)\subseteq\tau( \Agmat)\subseteq \AgPer$ and that the above diagram is commutative.
\end{proof}

\begin{lemma}\label{L:Per-Vor}
If $\sigma' \in \Per$ and $\sigma \in \Voro$ are such that  $\sigma'\subseteq \sigma$ then
$\sigma'\in \Mat$.
\end{lemma}
\begin{proof}
By Remark \ref{R:Per-simpli}, the extremal rays of $\sigma'$ are generated by rank one quadratic forms
$\{Q_1,\ldots,Q_m\}$; in particular we have that
\begin{equation}\label{E:convhull}
\sigma'={\rm conv}(Q_1,\ldots,Q_m),
\end{equation}
where ${\rm conv}$ denotes the  positive hull. Moreover, we can assume that $Q_i=v_i\cdot v_i^t$ for some primitive vector $v_i\in \Z^g$, uniquely determined up
to sign.


Consider now the quadratic form $\sum_i Q_i\in \Ort$.  Since $Q_i\in \sigma$ by assumption, from \cite[Prop. 3.3.5]{Val} we get that 
the Dirichlet-Voronoi polytope $\Vor(\sum_i Q_i)$ of the quadratic form
$\sum_i Q_i$ is the Minkowski sum of the Dirichlet-Voronoi polytopes $\Vor(Q_i)$ of the quadratic forms $Q_i$.
Since each $Q_i$ has rank one, $\Vor(Q_i)$ is a one dimensional segment for every $i=1,\ldots, m$. Therefore
$\Vor(\sum_i Q_i)$ is a zonotope and $\sigma_{\Del(\sum_i Q_i)}\in \Mat$ by Remark \ref{R:zonotope}. 


Explicitly, $\Del(\sum_i Q_i)$ is the generalized lattice dicing cut out by the central hyperplanes in $\R^g$ that define the normal fan 
of the zonotope $\Vor(\sum_i Q_i)$ (see \cite[Thm 7.16]{Zie}).
Since $\Vor(\sum_i Q_i)$ is the Minkowski sum of $\Vor(Q_i)$, it follows from \cite[Prop. 7.12]{Zie} that the normal fan of $\Vor(\sum_i Q_i)$ is the common refinement of the normal fans
of $Q_i$, each of which  is determined by the single hyperplane $H_i:=\{x\in \R^g\: : \: v_i^t \cdot x=0\}$. From Fact \ref{F:Mat-Voro}, we get that the matrix $A\in M_{g,m}(\Z)$ 
whose column vectors are $\{v_1,\ldots, v_m\}$ is a simple unimodular matrix and that $\sigma_{\Del(\sum_i Q_i)}=\sigma(A)$. Lemma \ref{L:prop-mat-cones} gives now that 
the extremal rays of $\sigma(A)$ are exactly those generated by the rank one quadratic forms $Q_i=v_i\cdot v_i^t$, which implies that 
\begin{equation}\label{E:conv3}
\sigma_{\Del(\sum_i Q_i)}=\sigma(A)= {\rm conv}(Q_1,\ldots,Q_m).
\end{equation}
By combining \eqref{E:convhull} and \eqref{E:conv3}, we get that $\sigma'=\sigma_{\Del(\sum_i Q_i)}\in \Mat$.

\end{proof}

\begin{remark}
The rational map $\tau:\AgVor\dashrightarrow \AgPer$ is defined  on an open subset which, in general, is strictly bigger than $\Agmat$. For example, if $g=4, 5$  it is known (see \cite{ER2} and the
references therein) that  $\Voro$ is a refinement of $\Per$ (i.e. every cone of $\Voro$ is contained in a cone of $\Per$), which is indeed equivalent to the fact that the map $\tau$ is defined everywhere;
on the other hand, it follows from Example \ref{small-dim} that if $g\geq 4$ then $\Agmat$ is strictly smaller than $\AgVor$.

Indeed, it was believed for a long period (the so-called Voronoi-Dickson hypothesis) that the map $\tau$ was defined everywhere, i.e. that $\Voro$ was a refinement of $\Per$ for any $g$ (see \cite{Vor} and \cite[p. 94]{NamT}). However, this was disproved for $g=6$ by Erdahl-Rybnikov (see \cite{ER2} and \cite{ER3}).
\end{remark}

\begin{remark}
As we mention earlier in this paper, there is another well-known admissible decomposition of $\Ort$, the {\bf central cone} decomposition $\Cen$  (see \cite[Sec. (8.9)]{NamT}).
The associated toroidal compactification of $\Ag$, called the central compactification of $\Ag$ and denoted by $\AgCen$, was shown by Igusa \cite{Igu} to be isomorphic to the normalization of
 the blow-up of the Satake compactification $\Ag^*$ of $\Ag$ along the boundary.
 The comparison of $\AgCen$ with the other two toroidal compactifications considered in this paper, namely $\AgPer$ and $\AgVor$, appears to be much less clear. For example, it has been proved
 by Alexeev-Brunyate \cite{AB}  that the Torelli map $\tg$ does \emph{not} extend to a regular map from $\Mgb$ to $\AgCen$ if $g\geq 9$ (while it does for $g\leq 8$, as shown in \cite{VIGRE}), thus disproving a long standing conjecture of Namikawa \cite{nam1}.
The proof of loc. cit. shows also that the rational map $\AgVor \dashrightarrow \AgCen$ is not regular on $\Agmat$ and, similarly, that $\AgPer\dashrightarrow \AgCen$ is not regular on $\tau(\Agmat)$.
\end{remark}

\emph{Acknowledgements.}

This project started during the trip back from the Workshop  ``Tropical and Non-Archimedean Geometry" held at
the Bellairs Research Institute (Barbados) during May 2011. We thank the organizers for inviting us to the Workshop
as well as the airline company
for providing us a very long and not too comfortable trip back to Europe.
We thank Valery Alexeev for some precious comments on an early draft of the paper. We are grateful to the referee for useful comments and in particular for 
pointing out a gap in a previous proof of Theorem \ref{T:Mat-cont-Perf} for the $R_{10}$ matroid.


\begin{thebibliography}{BLVSWZ99}



\bibitem{alex1}
V.  Alexeev: Complete moduli in the presence of semiabelian group action, Ann. of Math. 155 (2002) 611--708.

\bibitem{alex} V. Alexeev:  Compactified Jacobians and Torelli map, Publ. RIMS Kyoto Univ. 40 (2004) 1241--1265.

\bibitem{AB} V. Alexeev, A. Brunyate, Extending Torelli map to toroidal compactifications of Siegel space, to appear in Invent. Math., DOI: 10.1007/s00222-011-0347-2
(available at  arXiv:1102.3425).

\bibitem{VIGRE} V. Alexeev, R. Livingston, J. Tenini, M. Arap, X. Hu, L. Huckaba, P. Mcfaddin, S. Musgrave,
    J. Shin, C. Ulrich, Extended Torelli map to the Igusa blowup in genus $6$, $7$, and $8$, to appear in 
    Experimental Mathematics (available at arXiv:1105.4384v1).



\bibitem{AMRT} A. Ash, D. Mumford, M. Rapoport, Y. Tai,
Smooth compactification of locally symmetric varieties,
Lie Groups: History, Frontiers and Applications, Vol. IV. Math. Sci. Press,
Brookline, Mass., 1975.

\bibitem{BG} E. Baranovskii, V. Grishukhin, Non-rigidity degree of a lattice and rigid lattices,
European J. Combin.  22  (2001) 921--935.

\bibitem{BT} E. S. Barnes, D.W. Trenery, A class of extreme lattice-coverings of n-space by spheres,
J. Aust. Math. Soc. 14 (1972) 247--256.


\bibitem{BMV} S. Brannetti, M. Melo, F. Viviani, On the tropical Torelli map,
Adv. Math. 226 (2011) 2546--2586.

\bibitem{Bry} T. Brylawski, Modular constructions for combinatorial geometries, Trans. AMS 203 (1975) 1--44.





\bibitem{CV2} L. Caporaso, F. Viviani, Torelli theorem for stable curves,
J. Eur. Math. Soc. (JEMS) 13 (2011), no. 5, 1289--1329.


\bibitem{DG} V. Danilov, V. Grishukhin, Maximal unimodular
systems of vectors, Europ. J. Combinatorics 20 (1999) 507--526.

\bibitem{DM} P. Deligne, D. Mumford,  The irreducibility
of the space of curves of given genus, Inst. Hautes \'Etudes
Sci. Publ. Math. 36 (1969) 75--109.

\bibitem{DG0} M. Deza, V. Grishukhin, Nonrigidity degree of root lattices and their duals,
Geom. Dedicata  104  (2004) 15--24.



\bibitem{Dic} T. J. Dickson, On Voronoi reduction of positive
quadratic forms, J. Number Theory 4 (1972) 330--341.

\bibitem{Die}
R. Diestel, Graph theory, Graduate Text in Math. 173, Springer-Verlag, Berlin, 1997.

\bibitem{DV} M. Dutour, F. Vallentin, Some six-dimensional rigid forms, in the Proceedings of the Conference "Voronoi's Impact on Modern Science", Institute of Math.,
Kyiv 2005 (H. Syta, A. Yurachivsky, P. Engel eds.),  Proc. Inst. Math. Nat. Acad. Sci. Ukraine Vol. 55 (available at  arXiv:math/0401191).


\bibitem{EG} P. Engel, V. Grishukhin, An example of a non-simplicial $L$-type domain,
European J. Combin.  22  (2001) 491--496.


\bibitem{ER} R. M. Erdahl, S. S. Ryshkov, On lattice dicing,
European J. Combin.  15 (1994) 459--481.

\bibitem{ER2} R. Erdahl, K. Rybnikov,
Voronoi-Dickson hypothesis on perfect forms and L-types, preprint available at arXiv:math/0112097.

\bibitem{ER3} R. Erdahl, K. Rybnikov, On Voronoi's two tilings of the cone of metrical forms,
IV International Conference in "Stochastic Geometry, Convex Bodies, Empirical Measures $\&$ Applications to Engineering Science'', Vol. I (Tropea, 2001).
Rend. Circ. Mat. Palermo (2) Suppl. No. 70  (2002),  279--296.

\bibitem{FC} G. Faltings, C. L.  Chai,
Degeneration of abelian varieties,
With an appendix by David Mumford, Ergebnisse der Mathematik und ihrer Grenzgebiete (3) 22, Springer-Verlag, Berlin, 1990.

\bibitem{GKZ} I. Gelfand, M. Kapranov, A. V. Zelevinsky, Discriminants, resultants and multidimensional determinants. (Reprint of the 1994 edition). Modern Birkh\"auser Classics. Birkh\"auser Boston, Inc., Boston, MA, 2008.

\bibitem{Gib} A. Gibney, On extension of the Torelli map, preprint available at  arXiv:1104.3788.

\bibitem{Gwe} T. Gwena, Degenerations of cubic threefolds and matroids,  Proc. Amer. Math. Soc.  133  (2005) 1317--1323.


\bibitem{Koe} M. Koecher, Beitr\"age zu einer Reduktionstheorie in Positivit\"atsbereichen I - II,
Math. Ann. 141 (1960) 384--432, Math. Ann. 144 (1961) 175--182.

\bibitem{KZ} A. Korkine, G. Zolotarev, Sur le formes quadratiques positives, Math. Ann. 11 (1877) 242--292.

\bibitem{Igu} J. Igusa, A desingularization problem in theory of
Siegel modular functions, Math. Ann. 168 (1967) 228--260.

\bibitem{Mar} J. Martinet, Perfect lattices in Euclidean spaces, Grundlehren der Mathematischen Wissenschaften 327, Springer-Verlag, Berlin, 2003.

\bibitem{McM} P. McMullen, Space tilings zonotopes, Mathematika 22 (1975) 202-211.



\bibitem{Min} H. Minkowsky, Zur Theorie der positiven quadratischen Formen,
J. f\"ur Math. CI. 196--202 (1887).


\bibitem{nam1} Y.  Namikawa, On the canonical holomorphic map from the moduli space of stable curves to the Igusa monoidal
transform, Nagoya Math. Journal 52 (1973) 197--259.

\bibitem{nam2}  Y. Namikawa, A new compactification of the Siegel Space and Degenerations of Abelian Varieties I--II, Math. Ann. 221  (1976), 97--141,  201--241.

\bibitem{NamT} Y. Namikawa, Toroidal compactification of Siegel spaces, Lecture Notes in Mathematics  812, Springer, Berlin, 1980.

\bibitem{Ols} M. C. Olsson, Compactifying moduli spaces for abelian varieties, Lecture Notes in Mathematics, 1958. Springer-Verlag, Berlin, 2008.

\bibitem{Oxl} J. G. Oxley,  Matroid theory, Oxford Graduate Texts in Mathematics 3, Oxford Science Publications, 1992.

\bibitem{Sey} P. D. Seymour, Decomposition of regular matroids.  J. Combin. Theory Ser. B  28  (1980) 305--359.

\bibitem{She} G. C. Shephard, Space-filling zonotopes, Mathematika 21 (1974) 261--269.

\bibitem{SB} N. I. Shepherd-Barron, Perfect forms and the moduli
space of abelian varieties, Invent. Math.  163  (2006) 25--45.

\bibitem{Val} F. Vallentin, Sphere coverings, Lattices and Tilings
(in low dimensions), PhD Thesis, Technische Universit\"at M\"unchen 2003,
Available at  http://tumb1.ub.tum.de/publ/diss/ma/2003/vallentin.pdf


\bibitem{Vor} G. F. Voronoi,  Nouvelles applications des param\'etres continus \'a la th\'eorie de formes quadratiques - Deuxi\'eme
m\'emoire, J. f\"ur die reine und angewandte Mathematik 134
(1908) 198--287, 136 (1909) 67--178.


\bibitem{Zie} G. Ziegler, Lectures on polytopes, Graduate Texts in Mathematics 152, Springer-Verlag, New York, 1995.

\end{thebibliography}
\end{document}